\documentclass[reqno,12pt]{amsart}
\usepackage{epsfig,amscd,amssymb,amsmath,amsfonts}

\usepackage{amsmath}
\usepackage{graphicx}
\usepackage{lscape}
\usepackage{amsthm,color}
\usepackage{tikz}
\usepackage{multirow}
\usetikzlibrary{graphs}
\usetikzlibrary{graphs,quotes}
\usetikzlibrary{decorations.pathmorphing}

\tikzset{snake it/.style={decorate, decoration=snake}}
\tikzset{snake it/.style={decorate, decoration=snake}}

\usetikzlibrary{decorations.pathreplacing,decorations.markings,snakes}

\newtheorem{theorem}{Theorem}[section]
\newtheorem{lemma}[theorem]{Lemma}
\newtheorem{proposition}[theorem]{Proposition}
\theoremstyle{definition}
\newtheorem{definition}[theorem]{Definition}
\newtheorem{corollary}[theorem]{Corollary}

\newtheorem{example}[theorem]{Example}

\theoremstyle{remark}
\newtheorem{remark}[theorem]{Remark}

\numberwithin{equation}{section}

\newcommand{\aspace}{\\[0.5ex]}

\usepackage[margin=1.05in]{geometry}
\usepackage[colorlinks]{hyperref}
\usepackage{siunitx}

\makeatletter
\let\@wraptoccontribs\wraptoccontribs
\makeatother
\begin{document}
 \title{Existence of the Map $det^{S^3}$}\thanks{Appendix in collaboration with  Tony Passero}
 
\author{Steven R. Lippold}
\address{Department of Mathematics and Statistics, Bowling Green State University, Bowling Green, OH 43403}

\email{steverl@bgsu.edu}

\author{Mihai D. Staic}
\address{Department of Mathematics and Statistics, Bowling Green State University, Bowling Green, OH 43403 }
\address{Institute of Mathematics of the Romanian Academy, PO.BOX 1-764, RO-70700 Bu\-cha\-rest, Romania.}

\email{mstaic@bgsu.edu}
\subjclass{Primary 15A15; Secondary 05C65, 05C70, 15A75}

 \begin{abstract}
In this paper we show the existence of a nontrivial linear map $det^{S^3}:V_d^{\otimes\binom{3d}{3}}\to k$ with the property that $det^{S^3}(\otimes_{1\leq i<j<k\leq 3d}(v_{i,j,k}))=0$ if there exists $1\leq x<y<z<t\leq 3d$ such that $v_{x,y,z}=v_{x,y,t}=v_{x,z,t}=v_{y,z,t}$.   This gives a partial answer to a conjecture form \cite{S3}.  As an application, we use the map $det^{S^3}$ to study those $d$-partitions of the complete hypergraph $K_{3d}^3$ that have zero Betti numbers. 
We also discuss algebraic and combinatorial properties of a map $det^{S^r}:V_d^{\otimes\binom{rd}{r}}\to k$ which generalizes the determinant map, the map $det^{S^2}$ from \cite{dets2}, and  $det^{S^3}$.
 \end{abstract}

\keywords{Determinant, partitions of hypergraphs, Betti numbers}

 \maketitle

\section{Introduction}
 Let $V_d$ be a vector space of dimension $d$. The determinant is the unique, up to a constant, nontrivial linear map $det:V_d^{\otimes d}\to k$ with the property that $det(\otimes_{1\leq i\leq d}(v_i))=0$ when there is $1\leq x<y\leq d$ such that $v_x=v_y$. It is well known that the existence and uniqueness of the determinant map is equivalent with the fact that  $dim_k(\Lambda_{V_d}[d])=1$ (where $\Lambda_{V_d}[d]$ is the component of degree $d$ of the exterior algebra).

A generalization of the determinant was discussed in  \cite{dets2}, where it was shown that for every $d\geq 1$ there exists a nontrivial linear map $det^{S^2}:V_d^{\otimes\binom{2d}{2}}\to k$ with the property that $det^{S^2}(\otimes_{1\leq i<j\leq 2d}(v_{i,j}))=0$ if there is $1\leq x<y<z\leq 2d$ such that $v_{x,y}=v_{x,z}=v_{y,z}$. Further, when $d=2$ in \cite{sta2}, and $d=3$ in \cite{edge}, it was shown that $det^{S^2}$ is unique, up to a constant. As an application of the map $det^{S^2}$ to combinatorics,  it was shown in \cite{dets2} that if $(\Gamma_1,\Gamma_2,\dots,\Gamma_d)$ is an edge $d$-partition of the complete graph $K_{2d}$ then $(\Gamma_1,\Gamma_2,\dots,\Gamma_d)$ is cycle-free if and only if $det^{S^2}(f_{(\Gamma_1,\Gamma_2,\dots,\Gamma_d)})\neq 0$ (see Theorem \ref{thcfp} for more details). 

The main idea behind the map $det^{S^2}$ is a construction $\Lambda^{S^2}_{V_d}$, which was introduced in \cite{sta2} as a generalization of the exterior algebra $\Lambda_{V_d}$. Just like the determinant  $det$ is associated to $\Lambda_{V_d}[d]$ (the component of degree $d$ of the exterior algebra $\Lambda_{V_d}$), the map $det^{S^2}$ is associated to  $\Lambda^{S^2}_{V_d}[2d]$ (the component of degree $2d$ of a graded vector space $\Lambda^{S^2}_{V_d}$). 
Moreover, the existence and uniqueness of the map $det^{S^2}$ is equivalent with $dim_k(\Lambda^{S^2}_{V_d}[2d])=1$. 


In \cite{S3}, another generalization of the exterior algebra was introduced, the graded vector space $\Lambda^{S^3}_{V_d}$. It was shown that $dim_k\left(\Lambda^{S^3}_{V_2}[6]\right)=1$ and that there exists a unique, up to a constant, nontrivial linear map $det^{S^3}:V_2^{\otimes20}\to k$ with the property that
\begin{eqnarray}
det^{S^3}(\otimes_{1\leq i<j<k\leq 6}(v_{i,j,k}))=0, \label{detS3}
\end{eqnarray} 
if there exists $1\leq x<y<z<t\leq 6$ such that $v_{x,y,z}=v_{x,y,t}=v_{x,z,t}=v_{y,z,t}$. Concerning the results of \cite{S3}, it was conjectured that $dim_k\left(\Lambda^{S^3}_{V_d}[3d]\right)=1$ for all $d\geq 1$. 

In this paper, we show that for all $d\geq 1$ there exists a nontrivial linear map $det^{S^3}:V_d^{\otimes\binom{3d}{3}}\to k$  with the property that $det^{S^3}(\otimes_{1\leq i<j<k\leq 3d}(v_{i,j,k}))=0$ if there exists $1\leq x<y<z<t\leq 3d$ such that $v_{x,y,z}=v_{x,y,t}=v_{x,z,t}=v_{y,z,t}$. 
In particular this proves that $dim_k\left(\Lambda^{S^3}_{V_d}[3d]\right)\geq 1$, which partially answers the conjecture from \cite{S3}. As an application we use the map $det^{S^3}$ to give a characterization of those  $d$-partitions of the complete $3$-uniform hypergraph $K_{3d}^3$ that have zero Betti numbers. 
In \cite{S3} a more general construction $\Lambda^{S^r}_{V_d}$ was introduced and some connections to $d$-partitions of the complete $r$-uniform hypergraph $K_{rd}^r$ were discussed. It was conjectured that $dim_k(\Lambda^{S^r}_{V_d}[rd])=1$. Related to this question, we  construct a linear map $det^{S^r}:V_d^{\otimes\binom{rd}{r}}\to k$  and establish some of its algebraic and combinatorial properties.


The paper is organized  as follows:  In order to provide some context for our results and make the presentation self contained, in Section \ref{section2}, we recall certain properties of the determinant and the $det^{S^2}$ map. We present from \cite{dets2} the main construction  and some properties  of $det^{S^2}$. We also recall results from \cite{S3}, which include the definitions of $\Lambda^{S^3}_{V_d}$ and $\Lambda^{S^r}_{V_d}$. Lastly, we discuss briefly $d$-partitions of the $r$-uniform complete hypergraph $K_{n}^r$ and present a connection to $\Lambda^{S^r}_{V_d}$.

In Section  \ref{section3}, 
we start by introducing a collection of vector equations $\mathcal{E}_{m,n}$, and use the matrix associated to the system given by $\mathcal{E}_{m,n}$ for all $1\leq m<n<3d$ to define a map $det^{S^3}:V_d^{\otimes\binom{3d}{3}}\to k$. 
We show $det^{S^3}$ has property \ref{detS3}, 
and we introduce an element $E^{(3)}_d\in V_d^{\otimes\binom{3d}{3}}$, which we use to prove that $det^{S^3}$ is nontrivial. In particular, we obtain that $dim_k\left(\Lambda^{S^3}_{V_d}[3d]\right)\geq1$. Lastly, we prove that $det^{S^3}$  is invariant under the action of the group $SL_d(k)$.

In Section \ref{SectionCombS3} we show that if $\mathcal{P}=(\mathcal{H}_1, \mathcal{H}_2,\dots, \mathcal{H}_d)$ is a 
$d$-partition of the $3$-uniform hypergraph $K_{3d}^3$, and $\omega_{\mathcal{P}}$ is its corresponding element in   $V_d^{\otimes\binom{3d}{3}}$ then $det^{S^3}(\omega_{\mathcal{P}})\neq 0$ if and only if $\mathcal{P}$ is homogeneous and the Betti numbers $b_2(\mathcal{H}_i)$ are zero for all $1\leq i\leq d$. Along the way, we obtain an interesting combinatorial identity (Lemma \ref{lemmaCombrd}) which essentially is equivalent with the fact that the Euler characteristic of certain CW complexes is zero. 

In Section  \ref{section4}, we extend our construction for all $r\geq 1$ to give linear maps  $det^{S^r}:V_d^{\otimes\binom{rd}{r}}\to k$ with the property that $det^{S^r}(\otimes_{1\leq i_1<i_2<\ldots<i_r\leq rd}(v_{i_1,\ldots,i_r}))=0$ if there exists $1\leq x_1<x_2<\ldots<x_{r+1}\leq rd$ such that \[v_{x_1,x_2,\ldots,x_r}=v_{x_1,\ldots,x_{r-1},x_{r+1}}=v_{x_1,\ldots,x_{r-2},x_r,x_{r+1}}=\ldots=v_{x_1,x_3,\ldots,x_r,x_{r+1}}=v_{x_2,x_3,\ldots,x_{r+1}}.\]
When $r=1$, $r=2$ and $r=3$ one recovers the maps $det$, $det^{S^2}$ and  $det^{S^3}$ respectively.  In general it is not not clear if the maps $det^{S^r}$ are nontrivial. If true, this would show that $dim_k\left(\Lambda^{S^r}_{V_d}[rd]\right)\geq 1$. The uniqueness of the map $det^{S^r}$ (with the above property) is an open question.  We also discuss a combinatorial property of the map $det^{S^r}$ in terms of  homogeneous $d$-partitions of the complete $r$-uniform hypergraph $K_{rd}^r$ that have zero Betti numbers.  

In the Appendix we present some results obtained in MATLAB which check that the map $det^{S^r}$ is nontrivial for certain particular values of $(r,d)$. 
We also propose an element $E^{(r)}_d \in V_d^{\otimes\binom{rd}{r}}$ which we believe it is nonzero in $\Lambda^{S^r}_{V_d}[rd]$. The results in the Appendix were obtained in collaboration with Tony Passero. 

\section{Preliminaries}
\label{section2}

In this section we establish a few conventions and notations that will be used throughout this paper. In order to make the  presentation self contained we recall some notations and constructions from \cite{edge,sta2,dets2,S3}.

In this paper,  $k$ is a field of characteristic zero,  and $\otimes=\otimes_k$. We take $V_d$ to be a $k$-vector space of dimension $d$ with basis $\mathcal{B}_d=\{e_1,\dots,e_d\}$. Finally,  we denote $\binom{n}{k}=\frac{n!}{k!(n-k)!}$. 
\subsection{The maps $det$ and $det^{S^2}$}
Recall that the determinant map is the unique, up to a constant, nontrivial linear map $det:V_d^{\otimes d}\to k$ with the property that $det(\otimes_{1\leq i\leq d}(v_i))=0$ if there exists $1\leq x<y\leq d$ such that $v_x=v_y$. In particular, the determinant induces a unique, up to a constant, nontrivial map $\Lambda_{V_d}[d]\to k$ (where $\Lambda_{V_d}[d]$ is the component of degree $d$ of the  exterior algebra $\Lambda_{V_d}$). It is well known that the existence and uniqueness of the map $det$ with the above properties is equivalent to the fact that $dim_k\left(\Lambda_{V_d}[d]\right)=1$.

In  \cite{sta2}, a generalization $\Lambda^{S^2}_{V_d}$ of $\Lambda_{V_d}$ was introduced. The exact construction of the graded vector space $\Lambda^{S^2}_{V_d}$ is not important at this point, but it was conjectured in \cite{sta2} that $dim_k(\Lambda^{S^2}_{V_d}[2d])=1$. This conjecture is equivalent with the existence of a unique, up to a constant, nontrivial linear map 
$$det^{S^2}:V_d^{\otimes\binom{2d}{2}}\to k,$$ with the property that $det^{S^2}(\otimes_{1\leq i<j\leq 2d}(v_{i,j}))=0$ if there is $1\leq x<y<z\leq 2d$ such that $v_{x,y}=v_{x,z}=v_{y,z}$. The conjecture was checked in the cases $d=2$ in \cite{sta2}, and $d=3$ in \cite{edge}. Moreover, it was shown in \cite{dets2} that for every $d\geq 1$ there exists a nontrivial linear map $det^{S^2}:V_d^{\otimes\binom{2d}{2}}\to k$ with the property that $det^{S^2}(\otimes_{1\leq i<j\leq 2d}(v_{i,j}))=0$ if there is $1\leq x<y<z\leq 2d$ such that $v_{x,y}=v_{x,z}=v_{y,z}$. In particular, this shows that $dim_k\left(\Lambda^{S^2}_{V_d}[2d]\right)\geq1$ for all $d\geq1$. The uniqueness of the map $det^{S^2}$ is still an open question for $d\geq 4$. 

We recall the construction of $det^{S^2}$ from \cite{dets2}. Let $v_{i,j}\in V_d$ with $1\leq i<j\leq 2d$. For $1\leq m\leq 2d$, define the vector equations $\mathcal{E}_m$ by
\begin{equation}\label{dets2Rel}
    \sum_{s=1}^{m-1}(-1)^{s-1}v_{s,m}\lambda_{s,m}+\sum_{t=m+1}^{2d}(-1)^tv_{m,t}\lambda_{m,t}=0.
\end{equation}
As noted in \cite{dets2}, the equations $\mathcal{E}_s$ for $1\leq s\leq 2d$ are not independent, but they satisfy the relation
\begin{equation}\label{dets2Eq}
    \sum_{s=1}^{2d}(-1)^s\mathcal{E}_s=0.
\end{equation}

In particular this means that the equation $\mathcal{E}_{2d}$ is a consequence of the first $2d-1$ equations $\mathcal{E}_l$, where $1\leq l\leq 2d-1$. And so, when we study the system consisting of equations  $\mathcal{E}_{l}$ for $1\leq l\leq 2d$, we can ignore the equation $\mathcal{E}_{2d}$. 
 
Let $\mathcal{S}_{2d}((v_{i,j})_{1\leq i<j\leq 2d})$ be the system given by the vector equations $\mathcal{E}_l$, where $1\leq l<2d$ and $A_{2d}((v_{i,j})_{1\leq i<j\leq 2d})$ be the associated matrix. Notice that $A_{2d}((v_{i,j})_{1\leq i<j\leq 2d})$ is a square matrix of size $d(2d-1)\times d(2d-1)$.   We define the map $det^{S^2}:V_d^{\otimes\binom{2d}{2}}\to k$ given by
$$det^{S^2}(\otimes_{1\leq i<j\leq 2d}(v_{i,j}))=det(A_{2d}((v_{i,j})_{1\leq i<j\leq 2d})).$$ 
It was shown in \cite{dets2} that $det^{S^2}(\otimes_{1\leq i<j\leq 2d}(v_{i,j}))=0$ if there is $1\leq x<y<z\leq 2d$ such that $v_{x,y}=v_{x,z}=v_{y,z}$. Moreover,  if $E_d=(\otimes_{1\leq i<j\leq 2d}(e_{i,j}))$ is given by
\[e_{i,j}=\begin{cases}
e_t, \ \textrm{if $i<2t-1$, $i$ is odd and $j=2t-1$},\\
e_t, \ \textrm{if $i<2t$, $i$ is even and $j=2t$},\\
e_t, \ \textrm{if $i=2t-1$, $j>2t-1$ and $j$ is even},\\
e_t, \ \textrm{if $i=2t$, $j>2t$ and $j$ is odd},
\end{cases}\]
then $det^{S^2}(E_d)\neq0$.
Some geometrical properties of the map $det^{S^2}$ were studied in \cite{dets2,sv}

In addition to the geometric properties, there are combinatorial properties of $det^{S^2}$ related to edge $d$-partitions of the complete graph $K_{2d}$. It is known from \cite{edge} that the set of edge $d$-partitions of the complete graph $K_{2d}$ is in bijection with a basis of the vector space  $V_d^{\otimes \binom{2d}{2}}$. More precisely, to an edge $d$-partition $(\Gamma_1,\dots, \Gamma_d)$ of $K_{2d}$ one  associates the element $f_{(\Gamma_1,\dots, \Gamma_d)}=\otimes_{1\leq i<j\leq 2d}(f_{i,j})\in V_d^{\otimes \binom{2d}{2}}$ determined as follows: if the edge $(i,j)$ belongs to the subgraph $\Gamma_x$ of $K_{2d}$ then we define $f_{i,j}=e_x$. One can easily show that the set 
$$\mathcal{G}_{{\mathcal B}_d}^{S^2}[2d]=\{f_{(\Gamma_1,\dots, \Gamma_d)}| (\Gamma_1,\dots, \Gamma_d) ~{\rm is}~{\rm a}~d{\rm-partition}~{\rm of}~K_{2d}\},$$ is a basis for $V_d^{\otimes \binom{2d}{2}}$. Finally, we say that an edge $d$-partition $(\Gamma_1,\dots, \Gamma_d)$ of $K_{2d}$ is cycle-free if each graph  $\Gamma_x$  is cycle free.  The following result  was proven in \cite{dets2}.
\begin{theorem}\label{thcfp}
 Let $(\Gamma_1,\dots, \Gamma_d)$ be an edge $d$-partition of $K_{2d}$, and $f_{(\Gamma_1,\dots, \Gamma_d)}$ be the associated element in $V_d^{\otimes \binom{2d}{2}}$. Then, $(\Gamma_1,\dots, \Gamma_d)$ is cycle-free if and only if $det^{S^2}(f_{(\Gamma_1,\dots, \Gamma_d)})\neq 0$.
\end{theorem}
One should notice that a homogeneous $d$-partition $(\Gamma_1,\dots, \Gamma_d)$ is cycle free if and only if  $b_1(\Gamma_i)=0$ for all $1\leq i\leq d$, where $b_1(\Gamma)$ is the first Betti number of $\Gamma$. 

\subsection{The Graded Vector Space $\Lambda^{S^3}_{V_d}$}
In \cite{S3}, we introduced $\Lambda^{S^3}_{V_d}$, another generalization of the exterior algebra.   We recall now that construction.  

Let $\mathcal{T}^{S^3}_{V_d}$ be the graded vector space
$\mathcal{T}^{S^3}_{V_d}=\displaystyle\bigoplus_{n\geq0}\mathcal{T}^{S^3}_{V_d}[n],$
where \[\mathcal{T}^{S^3}_{V_d}[n]=V_d^{\otimes\binom{n}{3}}.\] 
Define  $\mathcal{E}_{V_d}^{S^3}[n]\subseteq\mathcal{T}^{S^3}_{V_d}[n]$ as the subspace linearly generated by simple tensors $\otimes_{1\leq i<j<k\leq n}(v_{i,j,k})$ with the property that there is $1\leq x<y<z<t\leq n$ such that $v_{x,y,z}=v_{x,y,t}=v_{x,z,t}=v_{y,z,t}$. We define $\Lambda^{S^3}_{V_d}=\displaystyle\bigoplus_{n\geq0}\Lambda^{S^3}_{V_d}[n],$ where
\[\Lambda^{S^3}_{V_d}[n]=\frac{\mathcal{T}^{S^3}_{V_d}[n]}{\mathcal{E}_{V_d}^{S^3}[n]}.\]
We summarize some relevant facts concerning $\Lambda^{S^3}_{V_2}[n]$ from \cite{S3} in the following proposition.
\begin{proposition}[\cite{S3}]\label{S3Facts}
 Let $V_2$ be a vector space of dimension 2. Then,
 \begin{enumerate}
     \item $dim_k\left(\Lambda^{S^3}_{V_2}[6]\right)=1$.
     \item There exists a unique, up to a constant, nontrivial linear map $det^{S^3}:V_2^{\otimes20}\to k$ with the property that $det^{S^3}(\otimes_{1\leq i<j<k\leq 6} (v_{i,j,k}))=0$ if there is $1\leq x<y<z<t\leq 6$ such that $v_{x,y,z}=v_{x,y,t}=v_{x,z,t}=v_{y,z,t}$.
     \item If $T:V_2\to V_2$, then $det^{S^3}(\otimes_{1\leq i<j<k\leq 6} (T(v_{i,j,k})))=det(T)^{10}det^{S^3}(\otimes_{1\leq i<j<k\leq 6} (v_{i,j,k}))$. In particular, $det^{S^3}$ is invariant under the action of $SL_2(k)$.
 \end{enumerate}
\end{proposition}
It was conjectured in \cite{S3} that $dim_k\left(\Lambda^{S^3}_{V_d}[3d]\right)=1$ for any $d\geq1$. This is equivalent with the fact that there is a unique, up to a constant, nontrivial linear map 
$$det^{S^3}:V_d^{\otimes\binom{3d}{3}}\to k,$$ with the property that $det^{S^3}(\otimes_{1\leq i<j<k\leq 3d}(v_{i,j,k}))=0$ if there is $1\leq x<y<z<t\leq 3d$ such that $v_{x,y,z}=v_{x,y,t}=v_{x,z,t}=v_{y,z,t}$. 

One of the  main results of this paper is the existence of such a nontrivial map $det^{S^3}$ for all $d\geq 1$. In particular, this implies that $dim_k\left(\Lambda^{S^3}_{V_d}[3d]\right)\geq 1$. The uniqueness of $det^{S^3}$ is still an open question. 

\subsection{The Graded Vector Space $\Lambda^{S^r}_{V_d}$}

 Next, we recall the construction for $\Lambda^{S^r}_{V_d}$ for $r\geq 1$. This construction was first described in \cite{S3} as a generalization of the exterior algebra, $\Lambda^{S^2}_{V_d}$, and $\Lambda^{S^3}_{V_d}$. 
 
 Let $\mathcal{T}^{S^r}_{V_d}[n]=V_d^{\otimes\binom{n}{r}},$
 and let the graded vector space $\mathcal{T}^{S^r}_{V_d}$ be defined as $$\mathcal{T}^{S^r}_{V_d}=\displaystyle\bigoplus_{n\geq0}\mathcal{T}^{S^r}_{V_d}[n].$$ 
Take $\mathcal{E}^{S^r}_{V_d}[n]\subseteq\mathcal{T}^{S^r}_{V_d}[n]$ be the subspace linearly generated by simple tensors $$(\otimes_{1\leq i_1<i_2<\ldots<i_r\leq n}(v_{i_1,\ldots,i_r}))\in \mathcal{T}^{S^r}_{V_d}[n],$$ with the property that there exists $1\leq x_1<x_2<\ldots<x_{r+1}\leq n$ such that 
\[v_{x_1,x_2,\ldots,x_r}=v_{x_1,\ldots,x_{r-1},x_{r+1}}=v_{x_1,\ldots,x_{r-2},x_r,x_{r+1}}=\ldots=v_{x_1,x_3,\ldots,x_r,x_{r+1}}=v_{x_2,x_3,\ldots,x_{r+1}}.\]
 Then $\Lambda^{S^r}_{V_d}$ is defined as
 $\Lambda^{S^r}_{V_d}=\bigoplus_{n\geq0}\Lambda^{S^r}_{V_d}[n],$ where \[\Lambda^{S^r}_{V_d}[n]=\frac{\mathcal{T}^{S^r}_{V_d}[n]}{\mathcal{E}^{S^r}_{V_d}[n]}.\] 
One can see that when $r=1$ we get the exterior algebra $\Lambda_{V_d}$, and when $r=2$ or $r=3$ we recover the constructions mentioned above.  It was conjectured in \cite{S3} that $dim_k\left(\Lambda^{S^r}_{V_d}[rd]\right)=1$. In Section \ref{section4} we give a possible path towards showing that  $dim_k\left(\Lambda^{S^r}_{V_d}[rd]\right)\geq1$. 
More precisely, we construct a linear map 
$$det^{S^r}:V_d^{\otimes\binom{rd}{r}}\to k,$$ with the property that $det^{S^r}(\otimes_{1\leq i_1<i_2<\ldots<i_r\leq rd}(v_{i_1,\ldots,i_r}))=0$ if there exists $1\leq x_1<x_2<\ldots<x_{r+1}\leq rd$ such that 
\[v_{x_1,x_2,\ldots,x_r}=v_{x_1,\ldots,x_{r-1},x_{r+1}}=v_{x_1,\ldots,x_{r-2},x_r,x_{r+1}}=\ldots=v_{x_1,x_3,\ldots,x_r,x_{r+1}}=v_{x_2,x_3,\ldots,x_{r+1}}.\]
It is not clear if this map is always nontrivial. 

In the Appendix we check that the map $det^{S^r}$ is nontrivial for certain values $4\leq r\leq 8$ and $2\leq d\leq 9$. Those results were obtained with MATLAB.

\subsection{Partition of Hypergraphs}

Lastly, we recall from \cite{bretto, S3} a few definitions and results about hypergraphs.  
\begin{definition} A hypergraph $\mathcal{H}=(V, E)$ consists of two finite sets  $V = \{v_1, v_2, \dots , v_n\}$ called  the set of vertices, and  $E=\{E_1, E_2, ... , E_m\}$  a family of subsets of $V$ called the hyperedges of $\mathcal{H}$. 

If every hyperedge of $\mathcal{H}$ is of size $r$, then $\mathcal{H}$ is called an $r$-uniform hypergraph.

For $2 \leq r \leq n$, we define the complete
$r$-uniform hypergraph to be the hypergraph $K^r_n= (V, E)$ for which $V=\{1,2,\dots ,n\}$, and  $E$
is the family of all subsets of $V$ of size $r$.
\end{definition} 
Notice that a $2$-uniform hypergraph is nothing else but a graph, and $K_n^2$ is the complete graph $K_n$.
\begin{definition} Let $\mathcal{H}$ be a hypergraph and $d\geq 2$ be a natural number. A $d$-partition of $\mathcal{H}$ is an ordered collection $\mathcal{P}=(\mathcal{H}_1,\mathcal{H}_2,...,\mathcal{H}_d)$ of sub-hypergraphs $\mathcal{H}_i$ of  $\mathcal{H}$ such that:\\
(i) $V(\mathcal{H}_i)=V(\mathcal{H})$ for all $1\leq i \leq d$,  \\
(ii) $E(\mathcal{H}_i)\cap E(\mathcal{H}_j)=\emptyset$ for all $i\neq j$, \\
(iii) $\cup_{i=1}^dE(\mathcal{H}_i)=E(\mathcal{H})$. 
\end{definition}
Notice that if $\mathcal{H}$ is a graph then we recover the definition of a $d$-partition of a graph from \cite{edge}. We denote by  $\mathcal{P}_d(K_n^r)$ the set of $d$-partition of the complete $r$-uniform hypergraph $K_n^r$.

Next we recall some notations from \cite{S3}. Let $\mathcal{B}_d=\{e_1,\dots,e_d\}$ be a basis for $V_d$. We define
\[\mathcal{G}_{\mathcal{B}_d}^{S^r}[n]=\{\otimes_{1\leq i_1<\dots<i_r\leq n}(v_{i_1,\dots,i_r})\in\mathcal{T}^{S^r}_{V_d}[n]\;\vert \; v_{i_1,\dots,i_r}\in  \mathcal{B}_d\}.\]
It is obvious that $\mathcal{G}_{\mathcal{B}_d}^{S^r}[n]$ is a basis for 
$\mathcal{T}^{S^r}_{V_d}[n]$, and so its image in $\Lambda^{S^r}_{V_d}[n]$ will be a system of generators. 

Moreover, there exists a bijection between the elements in $\mathcal{G}_{\mathcal{B}_d}^{S^r}[n]$ and the set of $d$-partitions of the hypergraph $K_n^r$. Indeed, to every element  $\omega=\otimes_{1\leq i_1<\dots<i_r\leq n}(v_{i_1,\dots,i_r})\in \mathcal{G}_{\mathcal{B}_d}^{S^r}[n]$ we associate the partition $\mathcal{P}_{\omega}=(\mathcal{H}_1,\dots,\mathcal{H}_d)$ where the hyperedge $\{i_1,\dots,i_r\}\in \mathcal{H}_i$ if and only if $v_{i_1,\dots,i_r}=e_i$. It is easy to see that this map is a bijection between
$\mathcal{G}_{\mathcal{B}_d}^{S^r}[n]$ and $\mathcal{P}_d(K_n^r)$.  For a $d$-partition $\mathcal{P}\in \mathcal{P}_d(K_n^r)$ the corresponding element in $\mathcal{G}_{\mathcal{B}_d}^{S^r}[n]$ will be denoted by $\omega_{\mathcal{P}}$.

\section{Existence of a Map $det^{S^3}$}\label{dets3Construct}
\label{section3}
The purpose of this section is to give for every vector space $V_d$ of dimension $d$, a nontrivial linear map $det^{S^3}:V_d^{\otimes\binom{3d}{3}}\to k$ with the property $det^{S^3}(\otimes_{1\leq i<j<k\leq 3d}(v_{i,j,k}))=0$ if there is $1\leq x<y<z<t\leq 3d$ such that $v_{x,y,z}=v_{x,y,t}=v_{x,z,t}=v_{y,z,t}$. In the case $d=2$, such a map was given in \cite{S3}. We also show that the map $det^{S^3}$ is invariant by the action of  $SL_d(k)$. 
\subsection{Construction of a Map $det^{S^3}$} We will follow a similar idea to the one from \cite{dets2}. 
 Let $v_{i,j,k}\in V_d$ for $1\leq i<j<k\leq 3d$. For each $1\leq m<n\leq 3d$, we define a vector equation $\mathcal{E}_{m,n}$ as
\begin{equation}\label{MainEq}
 \sum_{r=1}^{m-1}(-1)^{r-1}\lambda_{r,m,n}v_{r,m,n}+\sum_{s=m+1}^{n-1}(-1)^s\lambda_{m,s,n}v_{m,s,n}+\sum_{t=n+1}^{3d}(-1)^{t+1}\lambda_{m,n,t}v_{m,n,t}=0.
\end{equation}
By abuse of notation, we also denote the left side of Equation (\ref{MainEq}) as $\mathcal{E}_{m,n}$. Denote by $\mathcal{S}((v_{i,j,k})_{1\leq i<j<k\leq 3d})$ the system consisting of equations $\mathcal{E}_{m,n}$ for all $1\leq m<n\leq 3d$ and by  $A((v_{i,j,k})_{1\leq i<j<k\leq 3d})$ the associated matrix.

\begin{remark}\label{MatrixNotationRemark}
In order to talk about the associated matrix, one has to introduce an ordering for the set of indices $\{(i,j,k)~|~1\leq i<j<k\leq 3d\}$ and $\{(m,n)~|~1\leq m<n\leq 3d\}$ for the columns and rows of 
$A((v_{i,j,k})_{1\leq i<j<k\leq 3d})$ respectively. The order we choose will change our construction only up to a sign, so it's not really relevant. To make things simpler we consider the dictionary order on both these sets. 
\end{remark}

First, notice that  the vector equations defined in (\ref{MainEq}) are not linearly independent. More precisely we have the following.
\begin{lemma}\label{eqRelationsLemma}
    Let $1\leq n\leq 3d$. Then we have the relation $\mathcal{R}_n$. 
    \begin{equation}\label{relations}
     \mathcal{R}_n: \quad\quad\quad \sum_{m=1}^{n-1}(-1)^m\mathcal{E}_{m,n}+\sum_{p=n+1}^{3d}(-1)^{p+1}\mathcal{E}_{n,p}=0.
    \end{equation}
\end{lemma}
\begin{proof} In order to prove that relation $\mathcal{R}_n$ holds we will show that all terms that appear in $\mathcal{R}_n$ come in pair and with opposite sign. First notice that the vector $v_{x,y,z}$ appears in the relation $\mathcal{R}_n$ if  and only if $n\in\{x,y,z\}$. 

Case 1: If $1\leq i<j<n $ then the vector $v_{i,j,n}$ appears in equation $(-1)^{i}\mathcal{E}_{i,n}$ and equation $(-1)^{j}\mathcal{E}_{j,n}$ with coefficient $(-1)^i(-1)^{j}\lambda_{i,j,n}$ and $(-1)^j(-1)^{i-1}\lambda_{i,j,n}$ respectively, and so the two occurrence cancel each other. 

Case 2: If $1\leq i<n<j\leq 3d$ then the vector $v_{i,n,j}$ appears in equation $(-1)^{i}\mathcal{E}_{i,n}$ and equation $(-1)^{j+1}\mathcal{E}_{n,j}$ with coefficient $(-1)^i(-1)^{j+1}\lambda_{i,n,j}$ and $(-1)^{j+1}(-1)^{i-1}\lambda_{i,n,j}$ respectively, and so the two occurrence cancel each other. 

Case 3: If $ n<i<j\leq 3d$ then the vector $v_{n,i,j}$ appears in equation $(-1)^{i+1}\mathcal{E}_{n,i}$ and equation $(-1)^{j+1}\mathcal{E}_{n,j}$ with coefficient $(-1)^{i+1}(-1)^{j+1}\lambda_{n,i,j}$ and 
$(-1)^{j+1}(-1)^{i}\lambda_{n,i,j}$  respectively, and so the two occurrence cancel each other. 

\end{proof}

Lemma \ref{eqRelationsLemma} shows that when studying the system $\mathcal{S}((v_{i,j,k})_{1\leq i<j<k\leq 3d})$ it suffices to consider only the vector equations $\mathcal{E}_{m,n}$, where $1\leq m<n<3d$. Indeed, by Equation (\ref{relations}) we know that for each $1\leq n<3d$ we have
\begin{equation}\label{3dIsolated}
    (-1)^{3d}\mathcal{E}_{n,3d}=\sum_{m=1}^{n-1}(-1)^m\mathcal{E}_{m,n}+\sum_{p=n+1}^{3d-1}(-1)^{p+1}\mathcal{E}_{n,p}.
\end{equation}
So, we need not study $\mathcal{E}_{m,3d}$ for all $1\leq m<3d$ to study $\mathcal{S}((v_{i,j,k})_{1\leq i<j<k\leq 3d})$. 

\begin{definition}\label{mainDefs}
 Let $v_{i,j,k}\in V_d$ for all $1\leq i<j<k\leq 3d$. Let $\mathcal{S}_{3d}((v_{i,j,k})_{1\leq i<j<k\leq 3d})$ be the system consisting of the equations $\mathcal{E}_{m,n}$ for $1\leq m<n<3d$. Define $A_{3d}((v_{i,j,k})_{1\leq i<j<k\leq 3d})$ to be the $d\binom{3d-1}{2}\times d\binom{3d-1}{2}$ square matrix corresponding to the system $\mathcal{S}_{3d}((v_{i,j,k})_{1\leq i<j<k\leq 3d})$.

 Define a map $det^{S^3}:V_d^{\binom{3d}{3}}\to k$ by 
     \begin{equation}
         det^{S^3}((v_{i,j,k})_{1\leq i<j<k\leq 3d})=det(A_{3d}((v_{i,j,k})_{1\leq i<j<k\leq 3d})).
     \end{equation} 
\end{definition}

\begin{remark}
 From properties of determinants, it immediately follows that $det^{S^3}$ is multilinear. By abuse of notation, we will also write $det^{S^3}$ for the linear map $$det^{S^3}:V_d^{\otimes\binom{3d}{3}}\to k,$$ determined by 
 \begin{equation}
     det^{S^3}(\otimes_{1\leq i<j<k\leq 3d}(v_{i,j,k}))=det(A_{3d}((v_{i,j,k})_{1\leq i<j<k\leq 3d})).
 \end{equation}
\end{remark}

Here is the main property of the map $det^{S^3}$. 
\begin{lemma}\label{universality}
 The map $det^{S^3}$ has the property that $det^{S^3}(\otimes_{1\leq i<j<k\leq 3d}(v_{i,j,k}))=0$ if there exists $1\leq x<y<z<t\leq 3d$ such that $v_{x,y,z}=v_{x,y,t}=v_{x,z,t}=v_{y,z,t}.$
\end{lemma}
\begin{proof}
 Suppose there exists $1\leq x<y<z<t\leq 3d$ such that $v_{x,y,z}=v_{x,y,t}=v_{x,z,t}=v_{y,z,t}.$ 
Consider the matrix $A=A((v_{i,j,k})_{1\leq i<j<k\leq 3d})$. From the definition of $\mathcal{E}_{m,n}$, we have that
 \begin{align*}\label{BigMatrix}
  A=\begin{pmatrix}
                                         \dots&0&\dots&0&\dots&0&\dots&0&\dots\\
                                         \vdots&\vdots&\vdots&\vdots&\vdots&\vdots&\vdots&\vdots&\vdots\\
                                         \dots&(-1)^{z+1}v_{x,y,z}&\dots&(-1)^{t+1}v_{x,y,t}&\dots&0&\dots&0&\dots\\
                                         \vdots&\vdots&\vdots&\vdots&\vdots&\vdots&\vdots&\vdots&\vdots\\
                                         \dots&(-1)^yv_{x,y,z}&\dots&0&\dots&(-1)^{t+1}v_{x,z,t}&\dots&0&\dots\\
                                         \vdots&\vdots&\vdots&\vdots&\vdots&\vdots&\vdots&\vdots&\vdots\\
                                         \dots&0&\dots&(-1)^yv_{x,y,t}&\dots&(-1)^zv_{x,z,t}&\dots&0&\dots\\
                                         \vdots&\vdots&\vdots&\vdots&\vdots&\vdots&\vdots&\vdots&\vdots\\
                                         \dots&(-1)^{x-1}v_{x,y,z}&\dots&0&\dots&0&\dots&(-1)^{t+1}v_{y,z,t}&\dots\\		
                                         \vdots&\vdots&\vdots&\vdots&\vdots&\vdots&\vdots&\vdots&\vdots\\
                                         \dots&0&\dots&(-1)^{x-1}v_{x,y,t}&\dots&0&\dots&(-1)^zv_{y,z,t}&\dots\\
                                         \vdots&\vdots&\vdots&\vdots&\vdots&\vdots&\vdots&\vdots&\vdots\\
                                         \dots&0&\dots&0&\dots&(-1)^{x-1}v_{x,z,t}&\dots&(-1)^{y-1}v_{y,z,t}&\dots\\
                                         \vdots&\vdots&\vdots&\vdots&\vdots&\vdots&\vdots&\vdots&\vdots\\
                                        \end{pmatrix},
 \end{align*}
where we listed in order the rows corresponding to $\mathcal{E}_{x,y}$, $\mathcal{E}_{x,z}$, $\mathcal{E}_{x,t}$, $\mathcal{E}_{y,x}$, $\mathcal{E}_{y,t}$ and $\mathcal{E}_{z,t}$. 
Denote the column corresponding the $\lambda_{i,j,k}$ as $c_{i,j,k}$. Since $v_{x,y,z}=v_{x,y,t}=v_{x,z,t}=v_{y,z,t}$, we have that
\begin{equation}\label{S3Cols}
    (-1)^tc_{x,y,z}-(-1)^zc_{x,y,t}+(-1)^yc_{x,z,t}-(-1)^xc_{y,z,t}=0.
\end{equation}
Further, we have the columns of $A_{3d}((v_{i,j,k})_{1\leq i<j<k\leq 3d})$ still satisfy Equation (\ref{S3Cols}). Indeed, since $\mathcal{S}_{3d}((v_{i,j,k})_{1\leq i<j<k\leq 3d})$ is obtained from $\mathcal{S}((v_{i,j,k})_{1\leq i<j<k\leq 3d})$ by excluding the vector equations $\mathcal{E}_{m,3d}$ for $1\leq m<3d$, we have the columns of $A_{3d}((v_{i,j,k})_{1\leq i<j<k\leq 3d})$ still satisfy Equation (\ref{S3Cols}). So the columns of $A_{3d}((v_{i,j,k})_{1\leq i<j<k\leq 3d})$ are linearly dependent which implies that $det(A_{3d}((v_{i,j,k})_{1\leq i<j<k\leq 3d}))=0$. 

Thus, if $v_{x,y,z}=v_{x,y,t}=v_{x,z,t}=v_{y,z,t}$ then we have $$det^{S^3}(\otimes_{1\leq i<j<k\leq 3d}(v_{i,j,k}))=det(A_{3d}(v_{i,j,k})_{1\leq i<j<k\leq 3d})=0.$$
\end{proof}

\subsection{The Map $det^{S^3}$ is Nontrivial}

Next, we define an element $E^{(3)}_d\in V_d^{\otimes\binom{3d}{3}}$ and we show that $det^{S^3}(E^{(3)}_d)\neq 0$, which gives that $det^{S^3}$ is a nontrivial map.
\begin{definition}\label{defE3}
Take $d\geq 2$ and for every $1\leq a\leq d$ consider the set $$S_a=\{3a-2,3a-1,3a\}.$$ Let $1\leq i<j<k\leq 3d$ such that  $i\in S_a$, $j\in S_b$, and $k\in S_c$, where $1\leq a\leq b\leq c\leq d$. Define $e_{i,j,k}\in V_d$ by
\[e_{i,j,k}=\begin{cases}e_a, \ \ \ \ \textrm{if } i+j+k=0 \ \textrm{(mod $3$)},\\
e_b, \ \ \ \ \textrm{if } i+j+k=1 \ \textrm{(mod $3$)},\\
e_c, \ \ \ \ \textrm{if } i+j+k=2 \ \textrm{(mod $3$)}\end{cases}.\]

We define $E^{(3)}_d\in V_d^{\binom{3d}{3}}$ to be defined by $E^{(3)}_d=((e_{i,j,k})_{1\leq i<j<k\leq 3d})$. By abuse of notation, we will also denote the image of $E^{(3)}_d$ in $V_d^{\otimes\binom{3d}{3}}$ as $E^{(3)}_d$.
\end{definition} 
We have the following result.
\begin{lemma}\label{nontrivial}
 For every $d\geq 2$ we have $det^{S^3}(E^{(3)}_d)\neq0$. In particular, the map $det^{S^3}:V_d^{\otimes\binom{3d}{3}}\to k$ is nontrivial.
\end{lemma}
\begin{proof}
 First, we explain the strategy of our proof. Let $((\lambda_{i,j,k})_{1\leq i<j<k\leq 3d})$ be a solution for the system $\mathcal{S}((v_{i,j,k})_{1\leq i<j<k\leq 3d})$. Then,  $((\lambda_{i,j,k})_{1\leq i<j<k\leq 3d})$ is also a solution for the system $\mathcal{S}_{3d}((v_{i,j,k})_{1\leq i<j<k\leq 3d})$, as all of the vector equations which define $\mathcal{S}_{3d}((v_{i,j,k})_{1\leq i<j<k\leq 3d})$ are also used in $\mathcal{S}((v_{i,j,k})_{1\leq i<j<k\leq 3d})$. 

Conversely, by Equation (\ref{relations}), we know that for $1\leq p<3d$,
 \[\sum_{i=1}^{p-1}(-1)^i\mathcal{E}_{i,p}+\sum_{j=p+1}^{3d-1}(-1)^{j+1}\mathcal{E}_{p,j}=(-1)^{3d}\mathcal{E}_{p,3d}.\]
 So, if $((\lambda_{i,j,k})_{1\leq i<j<k\leq 3d})$ is a solution for the system $\mathcal{S}_{3d}((v_{i,j,k})_{1\leq i<j<k\leq 3d})$, then it is a solution for $\mathcal{S}((v_{i,j,k})_{1\leq i<j<k\leq 3d})$. This implies $((\lambda_{i,j,k})_{1\leq i<j<k\leq 3d})$ is a solution for $\mathcal{S}_{3d}((v_{i,j,k})_{1\leq i<j<k\leq 3d})$ if and only if $((\lambda_{i,j,k})_{1\leq i<j<k\leq 3d})$ is a solution for $\mathcal{S}((v_{i,j,k})_{1\leq i<j<k\leq 3d})$.

 By  the definition of the map $det^{S^3}$, the only solution of the system $\mathcal{S}_{3d}((v_{i,j,k})_{1\leq i<j<k\leq 3d})$ is the trivial solution if and only if  $det^{S^3}(\otimes_{1\leq i<j<k\leq 3d}(v_{i,j,k}))=det(A_{3d}((v_{i,j,k})_{1\leq i<j<k\leq 3d}))\neq 0$.  
Combining these remarks, it follows that in order to prove that $det^{S^3}(E^{(3)}_d)\neq 0$ it is enough to show that the only solution for $\mathcal{S}((e_{i,j,k})_{1\leq i<j<k\leq 3d})$ is the trivial solution.

 Let $((\lambda_{i,j,k})_{1\leq i<j<k\leq 3d})$ be a solution for $\mathcal{S}((e_{i,j,k})_{1\leq i<j<k\leq 3d})$. 
We will show that $\lambda_{i,j,k}=0$ for all $1\leq i<j<k\leq 3d$.

 \textbf{Case 1}: Let $i\in S_a$, $j\in S_b$, and $k\in S_c$, where $1\leq a<b<c\leq d$. 

First, suppose that $i+j+k=0 \ (mod \ 3)$. This means that $e_{i,j,k}=e_a$. In addition, we know that the vectors $e_{r,j,k}$ for all $1\leq r<j$ with $r\neq i$ belong to the set $\{e_1,\ldots,e_{a-1},e_{a+1},\ldots,e_d\}$. Lastly, we see that $e_{j,s,k}$ for $j<s<k$, and $e_{j,k,t}$ for $k<t\leq 3d$ all belong to the set $\{e_1,\ldots,e_{a-1},e_{a+1},\ldots,e_d\}$. By using the vector equation $\mathcal{E}_{j,k}$, this gives that
 \begin{align}
     \sum_{\substack{1\leq r\leq j-1, \\ r\neq i}}(-1)^{r-1}\lambda_{s,j,k}e_{s,j,k}&+(-1)^{i-1}\lambda_{i,j,k}e_a\nonumber\aspace
     +\sum_{j+1\leq s\leq k-1}(-1)^s\lambda_{j,s,k}e_{j,s,k}&+\sum_{k+1\leq t\leq 3d}(-1)^{t+1}\lambda_{j,k,t}e_{j,k,t}=0.
		\label{eqcase1}
 \end{align}
 Notice that in equation (\ref{eqcase1}) the vector  $e_a$ appears exactly once with coefficient $(-1)^{i-1}\lambda_{i,j,k}$, while the other vectors are in $\{e_1,\dots,e_d\}\setminus \{e_a\}$. Since $\{e_1,e_2,\ldots,e_d\}$ is a basis, we obtain that $\lambda_{i,j,k}=0$.
 
 Next suppose that $i+j+k=1 \ (mod \ 3)$, which means  that $e_{i,j,k}=e_b$. Notice that the vectors $e_{i,s,k}$ for all $i<s<k$ with $s\neq j$ belong to the set $\{e_1,\ldots,e_{b-1},e_{b+1},\ldots,e_d\}$. Lastly, we see that the vectors $e_{r,i,k}$ for $1\leq r< i$ and $e_{i,k,t}$ for $k<t\leq 3d$ all belong to the set $\{e_1,\ldots,e_{b-1},e_{b+1},\ldots,e_d\}$. Using the vector equation $\mathcal{E}_{i,k}$, this gives that
 \begin{align} \label{eqcase2}
     \sum_{1\leq r\leq i-1}(-1)^{r-1}\lambda_{s,i,k}e_{s,i,k}&+\sum_{\substack{i+1\leq s\leq k-1, \\ s\neq j}}(-1)^s\lambda_{i,s,k}e_{i,s,k}\nonumber\aspace 
     +(-1)^{j}\lambda_{i,j,k}e_b&+\sum_{k+1\leq t\leq 3d}(-1)^{t+1}\lambda_{i,k,t}e_{i,k,t}=0.
 \end{align}
 Notice that in equation (\ref{eqcase2}) the vector  $e_b$ appears exactly once with coefficient $(-1)^{j}\lambda_{i,j,k}$, while the other vectors are in $\{e_1,\dots,e_d\}\setminus \{e_b\}$. Since $\{e_1,e_2,\ldots,e_d\}$ is a basis, we obtain that $\lambda_{i,j,k}=0$.
 
 Lastly, suppose that $i+j+k=2 \ (mod \ 3)$, so $e_{i,j,k}=e_c$. The vectors $e_{i,j,t}$ for all $j<t\leq 3d$ with $t\neq k$ all belong to the set $\{e_1,\ldots,e_{c-1},e_{c+1},\ldots,e_d\}$. Also, the vectors $e_{r,i,j}$ for $1\leq r<i$ and $e_{i,s,j}$ for $i<s<j$ all belong to the set $\{e_1,\ldots,e_{c-1},e_{c+1},\ldots,e_d\}$. By using the vector equation $\mathcal{E}_{i,j}$, this gives that
 \begin{align} \label{eqcase3}
     \sum_{1\leq r\leq i-1}(-1)^{r-1}\lambda_{s,i,j}e_{s,i,j}&+\sum_{i+1\leq s\leq j-1}(-1)^s\lambda_{i,s,j}e_{i,s,j}\nonumber\aspace
     +\sum_{\substack{j+1\leq t\leq 3d, \\ t\neq k}}(-1)^{t+1}\lambda_{i,j,t}e_{i,j,t}&+(-1)^{k+1}\lambda_{i,j,k}e_c=0.
 \end{align}
 Notice that in equation (\ref{eqcase3}) the vector  $e_c$ appears exactly once with coefficient $(-1)^{k+1}\lambda_{i,j,k}$, while the other vectors are in $\{e_1,\dots,e_d\}\setminus \{e_c\}$. Since $\{e_1,e_2,\ldots,e_d\}$ is a basis, we obtain that $\lambda_{i,j,k}=0$.
To summarize, we proved that if $1\leq a<b<c\leq d$, $i\in S_a$, $j\in S_b$, and $k\in S_c$, then $\lambda_{i,j,k}=0$.
 
 \textbf{Case 2}: Let $1\leq a<b\leq d$. We will consider the case when $i\in S_a$, $k\in S_b$, and $j\in S_a$ or $j\in S_b$. 

 First, notice that in the equation $\mathcal{E}_{3a-2,3a-1}$, we have that $e_{3a-2,3a-1,3b-1}=e_b$ since $3a-2+3a-1+3b-1=-1 \ (mod \ 3)=2 \ (mod \ 3)$. In addition, all of the other vectors appearing in $\mathcal{E}_{3a-2,3a-1}$ belong to the set $\{e_1,\ldots,e_d\}\setminus \{e_b\}$. So, since $\{e_1,\ldots,e_d\}$ is a basis for $V_d$, we know $\lambda_{3a-2,3a-1,3b-1}=0$. Similarly, using the equations $\mathcal{E}_{3a-2,3a}$ and $\mathcal{E}_{3a-1,3a}$, we get $\lambda_{3a-2,3a,3b-2}=0$ and $\lambda_{3a-1,3a,3b}=0$, respectively.

 Next, notice that in the equation $\mathcal{E}_{3b-2,3b-1}$, we have that $e_{3a,3b-2,3b-1}=e_a$ since $3a+3b-2+3b-1=0 \ (mod \ 3)$. In addition, all of the other vectors appearing in $\mathcal{E}_{3b-2,3b-1}$ belong to the set $\{e_1,\ldots,e_d\}\setminus\{e_a\}$. So, since $\{e_1,\ldots,e_d\}$ is a basis for $V_d$, we now $\lambda_{3a,3b-2,3b-1}=0$. Similarly, using the equations $\mathcal{E}_{3b-2,3b}$ and $\mathcal{E}_{3b-1,3b}$, we get $\lambda_{3a-1,3b-2,3b}=0$ and $\lambda_{3a-2,3b-1,3b}=0$, respectively. The above is summarized in Table \ref{ZerosTable}.

 \begin{table}[h!!]
     \centering
     \begin{tabular}{|c|c|}
        \hline
        Equation  & Relation \\
        \hline
        &\\
        $\mathcal{E}_{3a-2,3a-1}$  & $\lambda_{3a-2,3a-1,3b-1}=0$\\
        &\\
        \hline
        &\\
        $\mathcal{E}_{3a-2,3a}$  & $\lambda_{3a-2,3a,3b-2}=0$\\
        &\\
        \hline
        &\\
        $\mathcal{E}_{3a-1,3a}$  & $\lambda_{3a-1,3a,3b}=0$\\
        &\\
        \hline
        &\\
        $\mathcal{E}_{3b-2,3b-1}$  & $\lambda_{3a,3b-2,3b-1}=0$\\
        &\\
        \hline
        &\\
        $\mathcal{E}_{3b-2,3b}$  & $\lambda_{3a-1,3b-2,3b}=0$\\
        &\\
        \hline
        &\\
        $\mathcal{E}_{3b-1,3b}$  & $\lambda_{3a-2,3b-1,3b}=0$\\
        &\\
        \hline
     \end{tabular}
     \vspace{1em}
     \caption{Equalities arising from the equations $\mathcal{E}_{x,y}$, where $x<y\in S_a$ or $x<y\in S_b$}
     \label{ZerosTable}
 \end{table}

 Next, if we consider equation $\mathcal{E}_{3a-2,3b-2}$, then due to Case 1 all of the terms are zero except for
 \begin{eqnarray*}(-1)^{3a-1}\lambda_{3a-2,3a-1,3b-2}e_{3a-2,3a-1,3b-2}+(-1)^{3a}\lambda_{3a-2,3a,3b-2}e_{3a-2,3a-3b-2}\\
 +(-1)^{3b}\lambda_{3a-2,3b-2,3b-1}e_{3a-2,3b-2,3b-1}+(-1)^{3b+1}\lambda_{3a-2,3b-2,3b}e_{3a-2,3b-2,3b}=0.
 \end{eqnarray*}
 From Definition \ref{defE3} we have $e_{3a-2,3a-1,3b-2}=e_a$ and $e_{3a-2,3a,3b-2}=e_{3a-2,3b-2,3b-1}=e_{3a-2,3b-2,3b}=e_b$. However, we already know that $\lambda_{3a-2,3a,3b-2}=0$ from Table \ref{ZerosTable}. This gives us two identities
 \begin{align}
     \lambda_{3a-2,3a-1,3b-2}=0,
 \end{align}
 and
 \begin{align}
     \lambda_{3a-2,3b-2,3b-1}=\lambda_{3a-2,3b-2,3b}.
 \end{align}
 Also, if we consider equation $\mathcal{E}_{3a-2,3b-1}$, then due to Case 1 all of the terms are zero except for
 \begin{eqnarray*}
(-1)^{3a-1}\lambda_{3a-2,3a-1,3b-1}e_{3a-2,3a-1,3b-1}+(-1)^{3a}\lambda_{3a-2,3a,3b-1}e_{3a-2,3a,3b-1}\\
+(-1)^{3b-2}\lambda_{3a-2,3b-2,3b-1}e_{3a-2,3b-2,3b-1}+(-1)^{3b+1}\lambda_{3a-2,3b-1,3b}e_{3a-2,3b-1,3b}=0.
\end{eqnarray*}
 From Definition \ref{defE3}, we have $e_{3a-2,3a,3b-1}=e_{3a-2,3b-1,3b}=e_a$ and $e_{3a-2,3a-1,3b-1}=e_{3a-2,3b-2,3b-1}=e_b$. But we know that $\lambda_{3a-2,3a-1,3b-1}=\lambda_{3a-2,3b-1,3b}=0$ from Table \ref{ZerosTable}. This gives us two identities
 \begin{align}
     \lambda_{3a-2,3a,3b-1}=0,
 \end{align}
 and
 \begin{align}
     \lambda_{3a-2,3b-2,3b-1}=0.
 \end{align}
 We can similarly use equations $\mathcal{E}_{3a-2,3b-1}$, $\mathcal{E}_{3a-2,3b}$, $\mathcal{E}_{3a-1,3b-2}$, $\mathcal{E}_{3a-1,3b-1}$, $\mathcal{E}_{3a-1,3b}$, $\mathcal{E}_{3a,3b-2}$, $\mathcal{E}_{3a,3b-1}$, and $\mathcal{E}_{3a,3b}$ to get the equalities given in Table \ref{Case2Table}.
 
    \begin{table}[h!!]
    \centering
        \begin{tabular}{|c|c|c|}
        \hline
          Equation&Equality Corresponding to $e_a$&Equality Corresponding to $e_b$  \\
          \hline
          &&\\
          $\mathcal{E}_{3a-2,3b-2}$&$\lambda_{3a-2,3a-1,3b-2}=0$&$\lambda_{3a-2,3b-2,3b-1}=\lambda_{3a-2,3b-2,3b}$\\
          &&\\
          \hline
          &&\\
          $\mathcal{E}_{3a-2,3b-1}$&$\lambda_{3a-2,3a,3b-1}=0$&$\lambda_{3a-2,3b-2,3b-1}=0$\\
          &&\\
          \hline
          &&\\
          $\mathcal{E}_{3a-2,3b}$&$\lambda_{3a-2,3a-1,3b}=\lambda_{3a-2,3a,3b}$&$\lambda_{3a-2,3b-2,3b}=0$\\
          &&\\
          \hline
          &&\\
          $\mathcal{E}_{3a-1,3b-2}$&$\lambda_{3a-2,3a-1,3b-2}=\lambda_{3a-1,3a,3b-2}$&$\lambda_{3a-1,3b-2,3b-1}=0$\\
          &&\\
          \hline
          &&\\
          $\mathcal{E}_{3a-1,3b-1}$&$\lambda_{3a-1,3a,3b-1}=0$&$\lambda_{3a-1,3b-2,3b-1}=\lambda_{3a-1,3b-1,3b}$\\
          &&\\
          \hline
          &&\\
          $\mathcal{E}_{3a-1,3b}$&$\lambda_{3a-2,3a-1,3b}=0$&$\lambda_{3a-1,3b-1,3b}=0$\\
          &&\\
          \hline
          &&\\
          $\mathcal{E}_{3a,3b-2}$&$\lambda_{3a-1,3a,3b-2}=0$&$\lambda_{3a,3b-2,3b}=0$\\
          &&\\
          \hline
          &&\\
          $\mathcal{E}_{3a,3b-1}$&$\lambda_{3a-2,3a,3b-1}=\lambda_{3a-1,3a,3b-1}$&$\lambda_{3a,3b-1,3b}=0$\\
          &&\\
          \hline
          &&\\
          $\mathcal{E}_{3a,3b}$&$\lambda_{3a-2,3a,3b}=0$&$\lambda_{3a,3b-2,3b}=\lambda_{3a,3b-1,3b}$\\
          &&\\
          \hline
     \end{tabular}
     \vspace{1em}
     \caption{Equalities arising from the vector equations $\mathcal{E}_{x,y}$ for $x\in S_a$ and $y\in S_b$}\label{Case2Table}
    \end{table}

		Let's first notice that all the entries in Table \ref{Case2Table} are zero. Indeed, for example from Equation $\mathcal{E}_{3a-2,3b}$ we know that $\lambda_{3a-2,3a-1,3b}=\lambda_{3a-2,3a,3b}$, while from equation Equation $\mathcal{E}_{3a-1,3b}$ we know that $\lambda_{3a-2,3a-1,3b}=0$ which implies that $\lambda_{3a-2,3a,3b}=0$. 
		
Similar arguments show that all entries in Table \ref{Case2Table} are equal to zero. One can then check by inspection of Tables \ref{ZerosTable}  and \ref{Case2Table} that $\lambda_{i,j,k}=0$ if $i\in S_a$, $k\in S_b$ and $j$ is either in $S_a$ or $S_b$ (there are 18 such entries).

 \textbf{Case 3}: Suppose $1\leq a\leq d$. We will consider the case $i,j,k\in S_a$, that is, $i=3a-2$, $j=3a-1$, and $k=3a$. Notice that $\lambda_{i,3a-2,3a-1}=\lambda_{3a-2,3a-1,k}=0$ for all $1\leq i<3a-2$ and $3a<k\leq 3d$ by Case 2. Further, $e_{3a-2,3a-1,3a}=e_a$. So, by using the vector equation $\mathcal{E}_{3a-2,3a-1}$, we get 
 \begin{equation}
     (-1)^{3a+1}\lambda_{3a-2,3a-1,3a}e_{3a-2,3a-1,3a}=0,
 \end{equation}
  thus, we have $\lambda_{3a-2,3a-1,3a}=0$.
 
 Therefore, we proved that the system $\mathcal{S}((e_{i,j,k})_{1\leq i<j<k\leq 3d})$ only has the trivial solution and so $det^{S^3}(E^{(3)}_d)\neq0$.
\end{proof}

To summarize we have the following theorem.

\begin{theorem}\label{mainresult}
 Let $d\geq2$ and $V_d$ be a $d$-dimensional vector space. The map $det^{S^3}:V_d^{\otimes\binom{3d}{3}}\to k$ is a nontrivial linear map with the property that $det^{S^3}(\otimes_{1\leq i<j<k\leq 3d}(v_{i,j,k}))=0$ if there exists $1\leq x<y<z<t\leq 3d$ such that $v_{x,y,z}=v_{x,y,t}=v_{x,z,t}=v_{y,z,t}$. 
\end{theorem}
\begin{proof}
Follows immediately from Lemmas \ref{universality} and \ref{nontrivial}.
\end{proof}
The next corollary immediately follows from Theorem \ref{mainresult} and the construction of $\Lambda^{S^3}_{V_d}$. It gives a partial answer to a conjecture from \cite{S3}.
\begin{corollary}\label{S3Dim}
 Let $V_d$ be a vector space of dimension $d$. Then, $dim_k\left(\Lambda^{S^3}_{V_d}[3d]\right)\geq1$.
\end{corollary}

\subsection{$SL_d(k)$ invariance}
Next we show that the map $det^{S^3}$ is invariant under the action of the group $SL_d(k)$. First notice that we can extend the action of the group  $GL_d(k)$ on $V_d$, to an action on $\mathcal{T}^{S^3}_{V_d}[3d]= V_d^{\otimes\binom{3d}{3}}$ by
\begin{equation}
  T\ast(\otimes_{1\leq i<j<k\leq 3d}(v_{i,j,k}))=\otimes_{1\leq i<j<k\leq 3d}(T(v_{i,j,k})).  
\end{equation}
We have the following.
\begin{proposition}\label{InvarProps}
    Let $T\in GL_d(k)$ and $v_{i,j,k}\in V_d$ for $1\leq i<j<k\leq 3d$. Then,
    \begin{equation}
        det^{S^3}(\otimes_{1\leq i<j<k\leq 3d}(T(v_{i,j,k})))=det(T)^{\binom{3d-1}{2}}\cdot det^{S^3}(\otimes_{1\leq i<j<k\leq 3d}(v_{i,j,k})).
    \end{equation}
    In particular, $det^{S^3}:V_d^{\otimes\binom{3d}{3}}\to k$ is invariant under the action by $SL_d(k)$.
\end{proposition}
\begin{proof}
    Let $T\in GL_d(k)$ be a linear transformation $T:V_d\to V_d$ and $v_{i,j,k}\in V_d$ for $1\leq i<j<k\leq 3d$. We can define the $d\binom{3d-1}{2}\times d\binom{3d-1}{2}$ matrix $I_{3,d,T}$ by the block matrix
    \[I_{3,d,T}=\begin{pmatrix}
        T&\textbf{0}&\ldots&\textbf{0}\\
        \textbf{0}&T&\ldots&\textbf{0}\\
        \vdots&\vdots&\ddots&\vdots\\
        \textbf{0}&\textbf{0}&\ldots&T
    \end{pmatrix},\]
    where \textbf{0} is the $d\times d$ zero matrix. Notice that $det(I_{3,d,T})=det(T)^{\binom{3d-1}{2}}$. 
 
With this notation, one can see that 
$$A_{3d}((T(v_{i,j,k}))_{1\leq i<j<k\leq 3d})=I_{3,d,T}A_{3d}((v_{i,j,k})_{1\leq i<j<k\leq 3d}).$$  We have
\begin{eqnarray*} det^{S^3}(T\ast((v_{i,j,k})_{1\leq i<j<k\leq 3d})))&=&det(A_{3d}((T(v_{i,j,k}))_{1\leq i<j<k\leq 3d}))\\
&=&det(I_{3,d,T}A_{3d}((v_{i,j,k})_{1\leq i<j<k\leq 3d}))\\
&=&det(I_{3,d,T})det(A_{3d}((v_{i,j,k})_{1\leq i<j<k\leq 3d}))\\
&=&det(T)^{\binom{3d-1}{2}}det^{S^3}((v_{i,j,k})_{1\leq i<j<k\leq 3d}), 
\end{eqnarray*}
which proves our statement. 
\end{proof}

\begin{remark}
   Since $det^{S^3}$ is $SL_d(k)$ invariant, from general results in invariant theory (see \cite{stu}, for example), we know that $det^{S^3}(\otimes_{1\leq i<j<k\leq 3d}(v_{i,j,k}))$ can be written as the sum of products of determinants of $d\times d$ matrices with columns corresponding to $v_{i,j,k}$.  In the case $d=2$ an explicit formula in terms of these determinants was given in \cite{S3}. It would be interesting to obtain an explicit formula in general.
\end{remark}

\section{$d$-partitions of $K_{3d}^3$ that have zero Betti numbers}

\label{SectionCombS3}


As mentioned in Section \ref{section2}, an edge  $d$-partition $(\Gamma_1,\Gamma_2,\dots, \Gamma_d)$ of the complete graph $K_{2d}$ is cycle free if and only if $det^{S^2}(f_{(\Gamma_1,\Gamma_2,\dots, \Gamma_d)})\neq 0$. 

Generalizing the idea of cycles from graphs to hypergraphs is not a trivial problem. For some results and accounts of this problem one can check \cite{geo,gk,rrs,seb}. A natural approach is suggested by the fact that a connected simple graph is cycle free if and only if its first Betti number is zero. With this remark, the above result shows that $det^{S^2}(f_{(\Gamma_1,\Gamma_2,\dots, \Gamma_d)})\neq 0$ if and only if  $(\Gamma_1,\Gamma_2,\dots, \Gamma_d)$ is a homogeneous $d$-partition of $K_{2d}$ and  $b_{1}(\Gamma_i)=0$  for all $1\leq i\leq d$, (where $b_{1}(\Gamma)$ is the first Betti number of the graph $\Gamma$). 

In this section we will show that if $(\mathcal{H}_1,\mathcal{H}_2,\dots, \mathcal{H}_d)$ is a $d$-partition of the complete hypergraph $K_{3d}^3$ then $det^{S^3}(\omega_{(\mathcal{H}_1,\mathcal{H}_2,\dots, \mathcal{H}_d)})\neq 0$ if and only if $(\mathcal{H}_1,\mathcal{H}_2,\dots, \mathcal{H}_d)$ is homogeneous and the 
$b_{2}(\mathcal{H}_i)=0$ for all $1\leq i\leq d$, where $b_{2}(\mathcal{H})$ is the second Betti number of the hypergraph $\mathcal{H}$.

First we introduce a few notations and conventions. In this section $\mathbb{R}$ is the field of real numbers,  and  $\{l_1,l_2,\dots,l_n\}$ is a fixed basis of the vector space $\mathbb{R}^n$.

If $\mathcal{H}=(V,E)$ is a hypergraph we denote 
\begin{eqnarray*}E_s(\mathcal{H})&=&\{\{a_1,a_2,\dots,a_s\}| \{a_1,a_2,\dots,a_s,x_1,x_2,\dots,x_{r-s}\}\in E(\mathcal{H})\\ &&~{\rm for ~ some ~}x_1,x_2,\dots,x_{r-s}\in V(\mathcal{H})\}.
\end{eqnarray*}

 For every sub-hypergraph $\mathcal{H}$ of $K_{n}^3$ one can associate a $2$-dimensional CW complex $X(\mathcal{H})\subseteq \mathbb{R}^n$ by taking 
$$X_{-1}=\emptyset,$$ 
$$X_{0}=\{l_{a}\in \mathbb{R}^n|\{a\}\in E_1(\mathcal{H})\},$$
 $$X_{1}=\{t_1l_{a}+t_2l_b\in \mathbb{R}^n|\;  t_1,t_2,\in [0,1],~t_1+t_2=1,~{\rm and }~\{a,b\}\in E_2(\mathcal{H})\},$$ and  
$$X_{2}=\{t_1l_{a}+t_2l_b+t_3l_c\in \mathbb{R}^n|\; t_1,t_2,t_3\in [0,1],~t_1+t_2+t_3=1,~{\rm and }~ \{a,b,c\}\in E_3(\mathcal{H}) \}.$$

\begin{definition} Let  $\mathcal{H}$ be a sub-hypergraph of $K_{n}^3$. We define the Betti numbers of $\mathcal{H}$   as $$b_i(\mathcal{H})=b_i(X(\mathcal{H})),$$ 
the  reduced Betti numbers of its corresponding CW complex $X(\mathcal{H})$. 
\end{definition}
\begin{remark} Notice that if $\mathcal{H}$ is a sub-hypergraph of $K_{n}^3$ then the CW complex $X(\mathcal{H})$ has dimension $2$, and so $b_{i}(\mathcal{H})=0$ for all $i\geq 3$. 
\end{remark}

\begin{example}
If we consider the hypergraph $K_{3}^3$, then $X(K_{3}^3)$ is the simplex $\Delta_2=\{(x,y,z)\in \mathbb{R}^3|\;x,y,z\in [0,1], ~x+y+z=1\}$. Its Betti numbers are $b_{-1}(K_{3}^3)=b_{0}(K_{3}^3)=b_{1}(K_{3}^3)=b_{2}(K_{3}^3)=0$.
\end{example}
\begin{example}
If we consider the hypergraph $K_{4}^3$, then $X(K_{4}^3)$ is the sphere $S^2\subseteq \mathbb{R}^4$. Its Betti numbers are $b_{-1}(K_{4}^3)=b_{0}(K_{4}^3)=b_{1}(K_{4}^3)=0$, and $b_{2}(K_{4}^3)=1$.
\end{example}

\begin{definition}
We say that a $d$-partition $(\mathcal{H}_1,\mathcal{H}_2,\dots, \mathcal{H}_d)$ of $K_{3d}^3$ is pre-homogeneous if  for all $1\leq i\leq d$ we have $$\vert E_1(\mathcal{H}_i)\vert=3d,$$ and  $$\vert E_2(\mathcal{H}_i)\vert=\frac{3d(3d-1)}{2}.$$ 
We say that a pre-homogeneous  $d$-partition $(\mathcal{H}_1,\mathcal{H}_2,\dots, \mathcal{H}_d)$ of $K_{3d}^3$ is homogeneous if  for all $1\leq i\leq d$ we have $$\vert E_3(\mathcal{H}_i)\vert=\frac{(3d-1)(3d-2)}{2}.$$
\end{definition}

\begin{remark}
Notice that $\vert E_1(K_{3d}^3)\vert=3d$, $\vert E_2(K_{3d}^3)\vert=\frac{3d(3d-1)}{2}$, and $\vert E_3(K_{3d}^3)\vert=d\frac{(3d-1)(3d-2)}{2}$. So, if $(\mathcal{H}_1,\mathcal{H}_2,\dots, \mathcal{H}_d)$ is a homogeneous $d$-partition of $K_{3d}^3$, then each  $X(\mathcal{H}_i)$ will have the same $1$-skeleton  as $X(K_{3d}^3)$, and the hyperedges of $K_{3d}^3$ are divided evenly among the hypergraphs $\mathcal{H}_i$. 
\end{remark}

\begin{lemma} Let $(\mathcal{H}_1,\mathcal{H}_2,\dots, \mathcal{H}_d)$ be a $d$-partition of $K_{3d}^3$ that is not pre-homogeneous, then $det^{S^3}(\omega_{(\mathcal{H}_1,\mathcal{H}_2,\dots, \mathcal{H}_d)})=0$. \label{lemmaPreH1} 
\end{lemma}
\begin{proof} 
First assume that there exist $1\leq m<n\leq 3d$, and $1\leq i\leq d$ such that $\{m,n\}\notin  E_2(\mathcal{H}_i)$. Consider the vector equation $\mathcal{E}_{m,n}$ and notice that none of the vectors $v_{s,m,n}$, $v_{m,t,n}$ or $v_{m,n,s}$ are equal to $e_i$. In particular, the $i$-th row of its corresponding  matrix $M_{m,n}(\omega_{(\mathcal{H}_1,\mathcal{H}_2,\dots, \mathcal{H}_d)})$ is zero. 

If $n\neq 3d$, then the matrix $M_{m,n}$ appears in $A_{3d}(\omega_{(\mathcal{H}_1,\mathcal{H}_2,\dots, \mathcal{H}_d)})$, which implies that  $det(A_{3d}(\omega_{(\mathcal{H}_1,\mathcal{H}_2,\dots, \mathcal{H}_d)}))=0$, and so $det^{S^3}(\omega_{(\mathcal{H}_1,\mathcal{H}_2,\dots \mathcal{H}_d)})=0$. 

If $n=3d$, then because $i$-th row of $M_{m,3d}$ is zero,  from relation $\mathcal{R}_m$ (see Lemma \ref{eqRelationsLemma}) we get that the $i$-th rows of the matrices $M_{s,m}$  and $M_{m,t}$ for $1\leq s<m$, and $m<t\leq 3d-1$ are linearly dependent. Since all of them appear in  $A_{3d}(\omega_{(\mathcal{H}_1,\mathcal{H}_2,\dots, \mathcal{H}_d)})$ we get again that $det(A_{3d}(\omega_{(\mathcal{H}_1,\mathcal{H}_2,\dots, \mathcal{H}_d)}))=0$, which implies that $det^{S^3}(\omega_{(\mathcal{H}_1,\mathcal{H}_2,\dots, \mathcal{H}_d)})=0$. 

If there exists a vertex $1\leq m\leq 3d$, and $1\leq i\leq d$ such that $m \notin  E_1(\mathcal{H}_i)$, then obviously any edge that ends in $m$ will not be in $E_2(\mathcal{H}_i)$ (for example $\{1,m\}$ or $\{m,3d\}$), and so we will be in the case we already discussed. 
\end{proof}

\begin{remark} \label{rem48} Let  $(\mathcal{H}_1,\mathcal{H}_2,\dots, \mathcal{H}_d)$ be a pre-homogeneous $d$-partition of $K_{3d}^3$.   We will denote by $\mathcal{K}_i$ the reduced-homology complex associated to the CW complex  $X(\mathcal{H}_i)$. More precisely, for all $1\leq i\leq d$ we denote 
\begin{itemize}
\item the generator in degree  $-1$ by $f_i^{\emptyset}$, 
\item the generators in degree $0$ by $f_{i}^{a}$ for $1\leq a\leq 3d$,  
\item the generators in degree $1$ by $f_{i}^{a,b}$ for $1\leq a<b\leq 3d$, 
\item the generators in degree $2$ by $f_{i}^{a,b,c}$ for $1\leq a<b<c\leq 3d$ such that $\{a,b,c\}\in E_3(\mathcal{H}_i)$.
\end{itemize} 
With this notation, the complex $\mathcal{K}_i$ is given by
\[ 0\to \bigoplus_{\{a,b,c\}\in E_3(\mathcal{H}_i)}kf_{i}^{a,b,c}\stackrel{\partial_2^i}{\longrightarrow} \bigoplus_{1\leq a<b\leq 3d}kf_{i}^{a,b} \stackrel{\partial_1^i}{\longrightarrow} \bigoplus_{1\leq a\leq 3d}kf_{i}^{a}\stackrel{\partial_0^i}{\longrightarrow} kf_{i}^{\emptyset}\to 0,\]
where $\partial_0^i(f_{i}^{a})=f_{i}^{\emptyset}$, $\partial_1^i(f_{i}^{a,b})=f_{i}^{b}-f_{i}^{a}$, and $\partial_2^i(f_{i}^{a,b,c})=f_{i}^{b,c}-f_{i}^{a,c}+f_{i}^{a,b}$.  One should notice that the complex $\mathcal{K}_i$ gives the reduced homology of the CW complex $X(\mathcal{H}_i)$. 
\end{remark}

\begin{lemma}  \label{lemmaPreH2} Let  $(\mathcal{H}_1,\mathcal{H}_2,\dots, \mathcal{H}_d)$ be a pre-homogeneous 
$d$-partition of $K_{3d}^3$. Then  $$rank (\bigoplus_{1\leq i\leq d}\partial_2^i)=rank (A(\omega_{(\mathcal{H}_1,\mathcal{H}_2,\dots, \mathcal{H}_d)})).$$
\end{lemma}
\begin{proof}
We will show that up to permuting rows and columns, and multiplying them by $-1$ the two matrices are the same. 
Indeed, first notice that $A(\omega_{(\mathcal{H}_1,\mathcal{H}_2,\dots, \mathcal{H}_d)})$ is a $d{3d \choose 2}\times {3d \choose 3}$ matrix, 
\[dim_k(\bigoplus_{1\leq i\leq d} (\bigoplus_{1\leq a<b\leq 3d}kf_{i}^{a,b}))=d{3d \choose 2},\]
and
\[dim_k(\bigoplus_{1\leq i\leq d}(\bigoplus_{\{a,b,c\}\in E_3(\mathcal{H}_i)}kf_{i}^{a,b,c}))={3d \choose 3},\]
so the corresponding matrices have the same dimension. 

The columns in both matrices are indexed by $\{a,b,c\}$ where $1\leq a<b<c\leq 3d$, so we can permute the columns of $\bigoplus_{1\leq i\leq d}\partial_2^i$ such that they correspond the the ordering used for $A(\omega_{(\mathcal{H}_1,\mathcal{H}_2,\dots, \mathcal{H}_d)})$ (see Remark \ref{MatrixNotationRemark}). 

For rows, we first  fix $\{a,b\}\in E_2(K_{3d}^3)$ and order  $\{f_i^{a,b}\}_{1\leq i\leq d}$ over the index $1\leq i\leq d$. If we denote such a collection  of rows by $f^{a,b}$, we can order them so it matches the order of the vector equations $\mathcal{E}_{a,b}$ in the matrix $A(\omega_{(\mathcal{H}_1,\mathcal{H}_2,\dots, \mathcal{H}_d)})$ (see Remark \ref{MatrixNotationRemark}).

Finally, notice that if in the matrix
 \[ A_{3d}((v_{i,j,k})_{1\leq i<j<k\leq 3d})=\bordermatrix{ &  & \mathcal{C}_{x,y,z} &  \cr
           & \vdots & 0 & \vdots \cr
       \mathcal{E}_{x,y} & \dots & (-1)^{z+1}v_{x,y,z}  & \dots \cr
       & \vdots & 0 & \vdots \cr
			   \mathcal{E}_{x,z}& \dots & (-1)^yv_{x,y,z}& \dots \cr
       & \vdots & 0 & \vdots \cr
			    \mathcal{E}_{y,z} &  \dots & (-1)^{x-1}v_{x,y,z}& \dots \cr
       & \vdots & 0 & \vdots  }, \qquad 
\]
we multiply the column $\mathcal{C}_{x,y,z}$ by $(-1)^{x+y+z+1}$, and the rows corresponding to the equation $\mathcal{E}_{x,y}$  by $(-1)^{x+y}$ we get the matrix 
 \[ \bordermatrix{ &  & \mathcal{C}_{x,y,z} &  \cr
           & \vdots & 0 & \vdots \cr
       \mathcal{E}_{x,y} & \dots & v_{x,y,z}  & \dots \cr
       & \vdots & 0 & \vdots \cr
			   \mathcal{E}_{x,z}& \dots & -v_{x,y,z}& \dots \cr
       & \vdots & 0 & \vdots \cr
			    \mathcal{E}_{y,z} &  \dots & v_{x,y,z}& \dots \cr
       & \vdots & 0 & \vdots  }. \qquad 
\]

Moreover, if  $\{x,y,z\}\in E_3(\mathcal{H}_j)$,  we have that $v_{x,y,z}=e_j$, and so this matrix becomes
\[ \bordermatrix{ &  & \mathcal{C}_{x,y,z} &  \cr
           & \vdots & 0 & \vdots \cr
       \mathcal{E}_{x,y} & \dots & e_j  & \dots \cr
       & \vdots & 0 & \vdots \cr
			   \mathcal{E}_{x,z}& \dots & -e_j& \dots \cr
       & \vdots & 0 & \vdots \cr
			    \mathcal{E}_{y,z} &  \dots & e_j& \dots \cr
       & \vdots & 0 & \vdots  }, \qquad 
\]
which can be identified with the the matrix associated to  $\bigoplus_{1\leq i\leq d}\partial_2^i$
\[ \bordermatrix{ &  & f_j^{x,y,z} &  \cr
           & \vdots & 0 & \vdots \cr
       f_j^{x,y} & \dots & 1  & \dots \cr
       & \vdots & 0 & \vdots \cr
			   f_j^{x,z}& \dots & -1& \dots \cr
       & \vdots & 0 & \vdots \cr
			    f_j^{y,z} &  \dots & 1& \dots \cr
       & \vdots & 0 & \vdots  }. \qquad 
\] 
In particular, we have that \[rank (\bigoplus_{1\leq i\leq d}\partial_2^i)=rank (A(\omega_{(\mathcal{H}_1,\mathcal{H}_2,\dots, \mathcal{H}_d)})).\]
\end{proof}

The following Lemma will be used in this section for $r=3$, and will play an important role when studding the combinatorial interpretation of the map $det^{S^r}$ discussed in Section \ref{section4}. Notice that the cases $d=1$ and $d=2$ are rather well known, while the case $r=3$ can be easily checked by direct computation.  The proof of the general result is elementary, but we were not able to find a reference.  
\begin{lemma} \label{lemmaCombrd}Let $r$ and $d$ be two positive integers. Then 
\begin{eqnarray}
\sum_{k=0}^{r-1}  (-1)^k\binom{rd}{k}+\frac{(-1)^r}{d}\binom{rd}{r}=0. \label{EquationRD}
\end{eqnarray}
\end{lemma}
\begin{proof}
Use the fact that $\binom{rd}{k}=\binom{rd-1}{k-1}+\binom{rd-1}{k}$, and $\frac{1}{d}\binom{rd}{r}=\binom{rd-1}{r-1}$.
\end{proof}

Next we state the main result of this section.
\begin{theorem} \label{ThComb} Let $(\mathcal{H}_1,\mathcal{H}_2,\dots, \mathcal{H}_d)$ be a $d$-partition of $K_{3d}^3$. The following are equivalent
\begin{enumerate} 
\item $det^{S^3}(\omega_{(\mathcal{H}_1,\mathcal{H}_2,\dots, \mathcal{H}_d)})\neq 0$.
\item  $(\mathcal{H}_1,\mathcal{H}_2,\dots \mathcal{H}_d)$ is pre-homogeneous, and  $b_k(\mathcal{H}_i)=0$ for every $1\leq i\leq d$, and $-1\leq k\leq 2$.
\item  $(\mathcal{H}_1,\mathcal{H}_2,\dots, \mathcal{H}_d)$ is pre-homogeneous, and  $b_2(\mathcal{H}_i)=0$ for every $1\leq i\leq d$.
\end{enumerate}
\end{theorem}
\begin{proof} 
Let's assume that $det^{S^3}(\omega_{(\mathcal{H}_1,\mathcal{H}_2,\dots \mathcal{H}_d)})\neq 0$, then by Lemma \ref{lemmaPreH1} we know that the partition $(\mathcal{H}_1,\mathcal{H}_2,\dots \mathcal{H}_d)$  must be pre-homogeneous. In particular, it follows that  for each $1\leq i\leq d$ the $1$-skeleton of $X(\mathcal{H}_i)$ coincide with the $1$-skeleton of $\Delta_{3d-1}$, and so we have that $b_{-1}(\mathcal{H}_i)=b_0(\mathcal{H}_i)=0$. 

Since $(\mathcal{H}_1,\mathcal{H}_2,\dots \mathcal{H}_d)$  is pre-homogeneous, to compute the homology of $X(\mathcal{H}_i)$ we can use the complex $\mathcal{K}_i$
\[ 0\to \bigoplus_{\{a,b,c\}\in E_3(\mathcal{H}_i)}kf_{i}^{a,b,c}\stackrel{\partial_2^i}{\longrightarrow} \bigoplus_{1\leq a<b\leq 3d}kf_{i}^{a,b} \stackrel{\partial_1^i}{\longrightarrow} \bigoplus_{1\leq a\leq 3d}kf_{i}^{a}\stackrel{\partial_0^i}{\longrightarrow} kf_{i}^{\emptyset}\to 0.\]

It follows from Lemma \ref{lemmaPreH2}
that $$rank (\bigoplus_{1\leq i\leq d}\partial_2^i)=rank (A(\omega_{(\mathcal{H}_1,\mathcal{H}_2,\dots, \mathcal{H}_d)}))=d\frac{(3d-1)(3d-2)}{2},$$ in other words the map $\bigoplus_{1\leq i\leq d}\partial_2^i$ is one to one.  This implies that each map $$\partial_2^i:\bigoplus_{\{a,b,c\}\in E_3(\mathcal{H}_i)}kf_{i}^{a,b,c}
\to \bigoplus_{1\leq a<b\leq 3d}kf_{i}^{a,b}$$ is one to one, in particular we get that $b_2(\mathcal{H}_i)=0$. 

Finally, notice that if we take $\mathcal{K}$ to be the direct sum $\bigoplus_{1\leq i\leq d}(\mathcal{K}_i)$ we get the complex 
\[ 0\to \bigoplus\limits_{i=1}^d(\bigoplus_{\{a,b,c\}\in E_3(\mathcal{H}_i)}kf_{i}^{a,b,c})\stackrel{\bigoplus\limits_{i=1}^d\partial_2^i}{\longrightarrow} \bigoplus\limits_{i=1}^d(\bigoplus_{1\leq a<b\leq 3d}kf_{i}^{a,b}) \stackrel{\bigoplus\limits_{i=1}^d\partial_1^i}{\longrightarrow} \bigoplus\limits_{i=1}^d(\bigoplus_{1\leq a\leq 3d}kf_{i}^{a})\stackrel{\bigoplus\limits_{i=1}^d\partial_0^i}{\longrightarrow}\bigoplus\limits_{i=1}^d(kf_{i}^{\emptyset})\to 0,\]
which is the complex associated to $X(K_{3d}^3)$, and whose Euler characteristic is $$\chi_{\mathcal{K}}=d-d(3d)+d\frac{3d(3d-1)}{2}-\frac{3d(3d-1)(3d-2)}{6}.$$ By Lemma \ref{lemmaCombrd} for $r=3$ we get  $\chi_{\mathcal{K}}=0$. 

On the other hand, we can use Betti numbers to compute 
$$\chi_{\mathcal{K}}=\sum_{i=1}^d(b_{-1}(\mathcal{H}_i)-b_{0}(\mathcal{H}_i)+b_{1}(\mathcal{H}_i)+b_{2}(\mathcal{H}_i)).$$We already proved that  $b_{-1}(\mathcal{H}_i)=b_{0}(\mathcal{H}_i)=b_{2}(\mathcal{H}_i)=0$ for all $1\leq i\leq d$, and so we get $$\sum_{i=1}^db_{1}(\mathcal{H}_i)=0.$$ Since Betti numbers are non-negative we have $b_{1}(\mathcal{H}_i)=0$ for all $1\leq i\leq d$.

Conversely, let's assume that $(\mathcal{H}_1,\mathcal{H}_2,\dots, \mathcal{H}_d)$ is pre-homogeneous and  $b_2(\mathcal{H}_i)=0$ for all $1\leq i\leq d$. 

Because the partition is pre-homogeneous, in order to compute $b_2(\mathcal{H}_i)$ we can use the morphism 
\[ 0\to \bigoplus_{\{a,b,c\}\in E_3(\mathcal{H}_i)}kf_{i}^{a,b,c}\stackrel{\partial_2^i}{\longrightarrow} \bigoplus_{1\leq a<b\leq 3d}kf_{i}^{a,b}.\]
Since $b_2(\mathcal{H}_i)=0$ for all $1\leq i\leq d$, it follows that the map $\partial_2^i$ is one to one for each $1\leq i\leq d$. This implies that the map $\bigoplus_{1\leq i\leq d}\partial_2^i$ is one to one, and so, by Lemma \ref{lemmaPreH2} we get that 
$$rank (A(\omega_{(\mathcal{H}_1,\mathcal{H}_2,\dots, \mathcal{H}_d)}))=rank (\bigoplus_{1\leq i\leq d}\partial_2^i)=d\frac{(3d-1)(3d-2)}{2},$$ which shows that  $det^{S^3}(\omega_{(\mathcal{H}_1,\mathcal{H}_2,\dots, \mathcal{H}_d)})\neq 0$. 
\end{proof}

\begin{corollary} Let $(\mathcal{H}_1,\mathcal{H}_2,\dots, \mathcal{H}_d)$ be a $d$-partition of $K_{3d}^3$ such that  $det^{S^3}(\omega_{(\mathcal{H}_1,\mathcal{H}_2,\dots, \mathcal{H}_d)})\neq 0$. Then $(\mathcal{H}_1,\mathcal{H}_2,\dots, \mathcal{H}_d)$ is a homogeneous $d$-partition.
\end{corollary} 
\begin{proof} From Theorem  \ref{ThComb}  we know that  $(\mathcal{H}_1,\mathcal{H}_2,\dots, \mathcal{H}_d)$ is pre-homogeneous and $b_k(\mathcal{H}_i)=0$ for all $1\leq i\leq d$, and $-1\leq k\leq 2$. In particular we have that the Euler characteristic $\chi_{\mathcal{K}_i}$ is zero. 

Since the partition $(\mathcal{H}_1,\mathcal{H}_2,\dots, \mathcal{H}_d)$  is pre-homogeneous we have that  $|E_1(\mathcal{H}_i)|=3d$, $|E_2(\mathcal{H}_i)|=\frac{3d(3d-1)}{2}$. From Remark \ref{rem48} we can compute the Euler characteristics $\chi_{\mathcal{K}_i}$ as 
$$\chi_{\mathcal{K}_i}=1-3d+\frac{3d(3d-1)}{2}-|E_3(\mathcal{H}_i)|.$$
By Lemma \ref{lemmaCombrd} for $r=3$ we get $$|E_3(\mathcal{H}_i)|=\frac{1}{d}\binom{3d}{3}=\frac{(3d-1)(3d-2)}{2},$$ in other words $(\mathcal{H}_1,\mathcal{H}_2,\dots, \mathcal{H}_d)$ is a homogeneous $d$-partition.
\end{proof}

\section{Construction and Properties of a Map $det^{S^r}$}
\label{section4}
 In this section, we present a construction for a map $det^{S^r}:V_d^{\otimes\binom{rd}{r}}\to k$ which generalizes the determinant map, the map $det^{S^2}$ given in \cite{dets2}, and the map $det^{S^3}$ given in Section 3. 
Specifically, we will construct a linear map $det^{S^r}:V_d^{\otimes\binom{rd}{r}}\to k$ with the property that $det^{S^r}((v_{i_1,\ldots,i_r})_{1\leq i_1<i_2<\ldots<i_r\leq rd})=0$ if there exists $1\leq x_1<x_2<\ldots<x_{r+1}\leq rd$ such that
 $$v_{x_1,x_2,\ldots,x_r}=v_{x_1,\ldots,x_{r-1},x_{r+1}}=v_{x_1,\ldots,x_{r-2},x_r,x_{r+1}}=\ldots=v_{x_1,x_3,\ldots,x_r,x_{r+1}}=v_{x_2,x_3,\ldots,x_{r+1}}.$$ 
Notice that if a nontrivial linear map with this property is constructed, then we have $dim_k\left(\Lambda^{S^r}_{V_d}[rd]\right)\geq1$ for all $r\geq 1$ and $d\geq 1$. This would partially answer a conjecture from \cite{S3}. It is not clear if the map we propose is always  nontrivial. In Appendix we analyze a few particular cases, and in all of them our map is nontrivial.   

In the second part of this section we  investigate a combinatorial property of the map $det^{S^r}$ similar to the one discussed in  Section \ref{SectionCombS3} for $det^{S^3}$.

 \subsection{Construction of $det^{S^r}$}
Fix $r\geq2$. For each $1\leq m_1<m_2<\ldots<m_{r-1}\leq rd$ and $v_{i_1,\dots,i_r}\in V_d$ for $1\leq i_1<i_2<\ldots<i_r\leq rd$, let the vector equation $\mathcal{E}_{m_1,\dots,m_{r-1}}$ be defined as
\begin{align}
    \sum_{s=1}^{m_1-1}(-1)^{s-1}&\lambda_{s,m_1,m_2,\dots,m_{r-1}}v_{s,m_1,m_2,\dots,m_{r-1}}+\sum_{s=m_1+1}^{m_2-1}(-1)^s\lambda_{m_1,s,m_2,\dots,m_{r-1}}v_{m_1,s,m_2\dots,m_{r-1}}+\dots\nonumber\\ \nonumber\\
    &+\sum_{s=m_{r-2}+1}^{m_{r-1}-1}(-1)^{(s-1)+(r-2)}\lambda_{m_1,m_2,\dots,s,m_{r-1}}v_{r,m_1,\dots,s,m_{r-1}}\label{relEqSr}\\ 
    \nonumber\\
    &+\sum_{s=m_{r-1}+1}^{rd}(-1)^{(s-1)+(r-1)}\lambda_{m_1,m_2,\dots,m_{r-1},s}v_{m_1,m_2\dots,m_{r-1},s}=0.\nonumber
\end{align}
By abuse of notation, we take $\mathcal{E}_{m_1,\dots,m_{r-1}}$ to also denote the left side of Equation (\ref{relEqSr}). Let $\mathcal{S}((v_{i_1,\dots,i_r})_{1\leq i_1<i_2<\dots<i_r\leq rd})$ be the system of equations given by $\mathcal{E}_{m_1,\dots,m_{r-1}}$ for all $1\leq m_1<\dots<m_{r-1}\leq rd$. Take $A((v_{i_1,\dots,i_r})_{1\leq i_1<i_2<\dots<i_r\leq rd})$ be the associated matrix.
\begin{remark}\label{NewMatRemark}
 Similar to Remark \ref{MatrixNotationRemark}, we order the rows and columns of $A((v_{i_1,\dots,i_r})_{1\leq i_1<i_2<\dots<i_r\leq rd})$ by the dictionary order on the indices of $\mathcal{E}_{m_1,\ldots,m_{r-1}}$ and $\lambda_{i_1,\ldots,i_r}$, respectively.
\end{remark}

Fix $n_1,\dots,n_{r-2}$. Similar to the case $r=2$ with Equation (\ref{dets2Eq}), and the case $r=3$ with Equation (\ref{relations}), we have the following relation $\mathcal{R}_{n_1,\dots,n_{r-2}}$:
\begin{equation}\label{Srrelations}
\begin{aligned}
    \sum_{s=1}^{n_1-1}(-1)^s\mathcal{E}_{s,n_1,\dots,n_{r-2}}&+\sum_{s=n_1+1}^{n_2-1}(-1)^{s+1}\mathcal{E}_{n_1,s,n_2,\dots,n_{r-2}}+\dots\aspace
    +\sum_{s=n_{r-3}+1}^{n_{r-2}-1}(-1)^{s+r-3}\mathcal{E}_{n_1,\dots,n_{r-3},s,n_{r-2}}&+\sum_{s=n_{r-2}+1}^{rd}(-1)^{s+r-2}\mathcal{E}_{n_1,\dots,n_{r-3},n_{r-2},s}=0
\end{aligned}
\end{equation}
In particular,  the equation $\mathcal{E}_{m_1,\ldots,m_{r-2},rd}$, for $1\leq m_1<m_2<\ldots<m_{r-2}<rd$,
is a consequence of the equations $\mathcal{E}_{n_1,\ldots,n_{r-2},n_{r-1}}$ for $1\leq n_1<n_2<\ldots<n_{r-1}<rd$. 
 So, we need not study the equations $\mathcal{E}_{m_1,\ldots,m_{r-2},rd}$ for all $1\leq m_1<m_2<\ldots<m_{r-2}<rd$ to study the system $\mathcal{S}((v_{i_1,\dots,i_r})_{1\leq i_1<i_2<\dots<i_r\leq rd})$. This leads to the following definition.

\begin{definition}\label{mainSrDefs}
 Let $r\geq 2$ and $v_{i_1,\dots,i_r}\in V_d$ for all $1\leq i_1<\dots<i_r\leq rd$. Take  the system $\mathcal{S}_{rd}((v_{i_1,\dots,i_r})_{1\leq i_1<\dots<i_r\leq rd})$ consisting from the  equations $\mathcal{E}_{m_1,m_2,\dots,m_{r-1}}$, where $1\leq m_1<m_2<\ldots<m_{r-1}<rd$. Denote the associated $d\binom{rd-1}{r-1}\times d\binom{rd-1}{r-1}$  matrix as $A_{rd}((v_{i_1,\dots,i_r})_{1\leq i_1<\dots<i_r\leq rd})$. Define the map $det^{S^r}:V_d^{\binom{rd}{r}}\to k$ by 
 \begin{equation}
     det^{S^r}((v_{i_1,\dots,i_r})_{1\leq i_1<\dots<i_r\leq rd})=det(A_{rd}((v_{i_1,\dots,i_r})_{1\leq i_1<\dots<i_r\leq rd})).
 \end{equation}
\end{definition}
\begin{remark} Note that when $r=2$, we get the $det^{S^2}$ map from \cite{dets2}, and when $r=3$, we get the map $det^{S^3}$ described in Section \ref{dets3Construct}.
\end{remark}
\begin{remark}
    The map $det^{S^r}$ is multilinear by properties of determinants. By abuse of notation, we will also denote the induced linear map $V_d^{\otimes\binom{rd}{r}}\to k$ by $det^{S^r}$.
\end{remark}

\subsection{Properties of $det^{S^r}$}

First, we generalize Lemma \ref{universality}.

\begin{lemma}\label{SrUniversality}
 The map $det^{S^r}$ has the property that $det^{S^r}((v_{i_1,\ldots,i_r})_{1\leq i_1<i_2<\ldots<i_r\leq rd})=0$ if there exists $1\leq x_1<x_2<\ldots<x_{r+1}\leq rd$ such that
 $$v_{x_1,x_2,\ldots,x_r}=v_{x_1,\ldots,x_{r-1},x_{r+1}}=v_{x_1,\ldots,x_{r-2},x_r,x_{r+1}}=\ldots=v_{x_1,x_3,\ldots,x_r,x_{r+1}}=v_{x_2,x_3,\ldots,x_{r+1}}.$$
\end{lemma}
\begin{proof}
The proof is similar with the one for Lemma \ref{universality}. Take $(v_{i_1,\ldots,i_r})_{1\leq i_1<i_2<\ldots<i_r\leq rd}\in V_d^{{\binom{rd}{r}}}$ such that 
\begin{equation}
    v_{x_1,x_2,\ldots,x_r}=v_{x_1,\ldots,x_{r-1},x_{r+1}}=v_{x_1,\ldots,x_{r-2},x_r,x_{r+1}}=\ldots=v_{x_1,x_3,\ldots,x_r,x_{r+1}}=v_{x_2,x_3,\ldots,x_{r+1}}=v,
\end{equation}
for some $1\leq x_1<x_2<\ldots<x_{r+1}\leq rd$. Consider the system $\mathcal{S}((v_{i_1,\ldots,i_r})_{1\leq i_1<i_2<\ldots<i_r\leq rd})$.

We note that the columns of $A((v_{i_1,\ldots,i_r})_{1\leq i_1<i_2<\ldots<i_r\leq rd})$ are linearly dependent. Indeed, if we denote the column corresponding to $\lambda_{i_1,\ldots,i_r}$ in the matrix $A((v_{i_1,\ldots,i_r})_{1\leq i_1<i_2<\ldots<i_r\leq rd})$ as $c_{i_1,\ldots,i_r}$, then we have that
\begin{align}
    (-1)^{x_{r+1}}c_{x_1,x_2,\ldots,x_{r-1},x_r}&-(-1)^{x_r}c_{x_1,x_2,\ldots,x_{r-1},x_{r+1}}+\ldots\nonumber\\
    &+(-1)^{r-1}(-1)^{x_2}c_{x_1,x_3,x_4,\ldots,x_{r+1}}+(-1)^{(r+1)-1}(-1)^{x_1}c_{x_2,x_3,\ldots,x_{r+1}}=0.\label{SrMatrixRel}
\end{align}
Further, since $\mathcal{S}_{rd}((v_{i_1,\ldots,i_r})_{1\leq i_1<\ldots<i_r\leq rd})$ is obtained by excluding vector equations from $\allowbreak\mathcal{S}((v_{i_1,\ldots,i_r})_{1\leq i_1<\ldots<i_r\leq rd})$, we have that the columns of $A_{rd}((v_{i_1,\ldots,i_r})_{1\leq i_1<i_2<\ldots<i_r\leq rd})$ are linearly dependent as well. Hence $$det^{S^r}((v_{i_1,\ldots,i_r})_{1\leq i_1<i_2<\ldots<i_r\leq rd})=det\left(A_{rd}((v_{i_1,\ldots,i_r})_{1\leq i_1<i_2<\ldots<i_r\leq rd})\right)=0.$$ 
\end{proof}

Next, notice that  there is a $GL_d(k)$ action on $\mathcal{T}^{S^r}_{V_d}[rd]$ given by
\[T\ast(\otimes_{1\leq i_1<\ldots<i_r\leq rd}(v_{i_1,\ldots,i_r}))=\otimes_{1\leq i_1<\ldots<i_r\leq rd}(T(v_{i_1,\ldots,i_r})).\]
This action leads to the following result, which has a similar proof as Proposition \ref{InvarProps}.
\begin{lemma}\label{SrSLInvar}
    Let $r\geq1$, $T\in GL_d(k)$ and $v_{i_1,i_2,\ldots,i_r}\in V_d$ for $1\leq i_1<\ldots<i_r\leq rd$. Then,
    \begin{equation}\label{SrGLRel}
        det^{S^r}(\otimes_{1\leq i_1<\ldots<i_r\leq rd}(T(v_{i_1,i_2,\ldots,i_r})))=det(T)^{\binom{rd-1}{r-1}}det^{S^r}(\otimes_{1\leq i_1<\ldots<i_r\leq rd}(v_{i_1,i_2,\ldots,i_r})).
    \end{equation}
    In particular, $det^{S^r}$ is invariant under the action of $SL_d(k)$.
\end{lemma}

\begin{remark}
    Since $det^{S^r}$ is multilinear and $SL_d(k)$ invariant, it would be interesting to see $det^{S^r}(\otimes_{1\leq i_1<\ldots<i_r\leq rd}(v_{i_1,i_2,\ldots,i_r}))$ written as a sum of products of determinants of $d\times d$ matrices whose columns are give by $v_{i_1,i_2,\dots, i_r}$. 
\end{remark}

\subsection{The $d$-partitions of $K_{rd}^r$ that have zero Betti Numbers}

\label{SectionCombSR}
Next, we investigate a combinatorial property of the map $det^{S^r}$. The results are presented in parallel with those for the case $r=3$ from Section \ref{SectionCombS3}.

We will show that if $(\mathcal{H}_1,\mathcal{H}_2,\dots, \mathcal{H}_d)$ is a $d$-partition of the complete hypergraph $K_{rd}^r$ then $det^{S^r}(\omega_{\mathcal{H}_1,\mathcal{H}_2,\dots, \mathcal{H}_d})\neq 0$ if and only if $(\mathcal{H}_1,\mathcal{H}_2,\dots, \mathcal{H}_d)$ is pre-homogeneous and the 
$b_{r-1}(\mathcal{H}_i)=0$ for all $1\leq i\leq d$.

Just like in Section \ref{SectionCombS3}, take $\{l_1,l_2,\dots,l_n\}$ to be a basis for $\mathbb{R}^n$.  To a sub-hypergraph $\mathcal{H}$ of $K_{n}^r$ one can associate an $(r-1)$-dimensional CW complex $X(\mathcal{H})\subseteq \mathbb{R}^n$, where for each $0\leq s\leq r$ we take
\begin{eqnarray*}
X_{s-1}&=&\{t_1l_{a_1}+t_2l_{a_2}+\dots +t_{s}l_{a_s}|\; t_1,t_2,\dots,t_s\in [0,1],~t_1+t_2+\dots+t_s=1,~{\rm and }~ \\
&&\{a_1,a_2,\dots,a_s\}\in E_s(\mathcal{H})\}.
\end{eqnarray*}

\begin{definition} Let  $\mathcal{H}$ be a sub-hypergraph of $K_{n}^r$. We define the Betti numbers of $\mathcal{H}$   as $b_i(\mathcal{H})=b_i(X(\mathcal{H}))$ (i.e. the  reduced Betti numbers of its corresponding CW complex $X(\mathcal{H})$). 
\end{definition}
\begin{remark} Notice that if $\mathcal{H}$ is a sub-hypergraph of $K_{n}^r$ then the CW complex $X(\mathcal{H})$ has dimension $r-1$, and so $b_{i}(\mathcal{H})=0$ for all $i\geq r$. 
\end{remark}

\begin{definition}
We say that a $d$-partition $(\mathcal{H}_1,\mathcal{H}_2,\dots, \mathcal{H}_d)$ of $K_{rd}^r$ is pre-homogeneous if  $$\vert E_k(\mathcal{H}_i)\vert=\binom{rd}{k}$$ for all $1\leq i\leq d$  and $1\leq k\leq r-1$.

We say that a pre-homogeneous  $d$-partition $(\mathcal{H}_1,\mathcal{H}_2,\dots, \mathcal{H}_d)$ of $K_{rd}^r$ is homogeneous if  for all $1\leq i\leq d$ we have $$\vert E_r(\mathcal{H}_i)\vert=\binom{rd-1}{r-1}=\frac{1}{d}\binom{rd}{r}.$$
\end{definition}

\begin{remark}
Notice that $\vert E_k(K_{rd}^r)\vert=\binom{rd}{k}$ for all $1\leq k\leq r-1$,  and $\vert E_r(K_{rd}^r)\vert=d\binom{rd-1}{r-1}$. So, in a homogeneous $d$-partition of $K_{rd}^r$, each  $\mathcal{H}_i$ will have the same $r-2$-skeleton  as $K_{rd}^r$, and the hyperedges of $K_{rd}^r$ will be divided evenly among the hypergraphs $\mathcal{H}_i$. 
\end{remark}

\begin{lemma} Let $(\mathcal{H}_1,\mathcal{H}_2,\dots, \mathcal{H}_d)$ be a $d$-partition of $K_{rd}^r$ that is not pre-homogeneous, then $det^{S^r}(\omega_{(\mathcal{H}_1,\mathcal{H}_2,\dots, \mathcal{H}_d)})=0$. \label{lemmaPreH1R} 
\end{lemma}
\begin{proof} 
The proof is similar with that for Lemma \ref{lemmaPreH1} 
\end{proof}

\begin{remark}  Let  $(\mathcal{H}_1,\mathcal{H}_2,\dots, \mathcal{H}_d)$ be a pre-homogeneous $d$-partition of $K_{rd}^r$.   We will denote by $\mathcal{K}_i$ the reduced-homology complex associated to the CW complex  $X(\mathcal{H}_i)$. More precisely, for all $1\leq i\leq d$ and $0\leq k\leq r-2$ we  denote 
\begin{itemize}
\item for $0\leq k\leq r-1$ we denote the generators in degree  $k-1$  by  $f_i^{a_1,a_2,\dots,a_k}$ for $1\leq a_1<a_2<\dots<a_k\leq rd$,
\item we denote the generators in degree $r-1$ by $f_{i}^{a_1,a_2,\dots,a_r}$ for $1\leq a_1<a_2<\dots<a_r\leq rd$ such that $\{a_1,a_2,\dots,a_r\}\in E_{r}(\mathcal{H}_i)$.
\end{itemize} With this notation the complex $\mathcal{K}_i$ is given by
\[ 0\to \bigoplus_{\{a_1,\dots,a_r\}\in E_{r}(\mathcal{H}_i)}kf_{i}^{a_1,\dots,a_r}\stackrel{\partial_{r-1}^i}{\longrightarrow} \bigoplus_{1\leq a_1<\dots<a_{r-1}\leq rd}kf_{i}^{a_1,\dots,a_{r-1}} \stackrel{\partial_{r-2}^i}{\longrightarrow} \dots \bigoplus_{1\leq a_1\leq rd}kf_{i}^{a_1}\stackrel{\partial_0^i}{\longrightarrow} kf_{i}^{\emptyset}\to 0\]
where  $$\partial_k^i(f_{i}^{a_1,a_2,\dots,a_k})=f_{i}^{a_2,a_3,\dots,a_k}-f_{i}^{a_1,a_3,\dots,a_k}+\dots+
(-1)^{k-1}f_{i}^{a_1,a_2,\dots,a_{k-1}}.$$   
\end{remark}

\begin{lemma}  \label{lemmaPreH2R} Let  $(\mathcal{H}_1,\mathcal{H}_2,\dots, \mathcal{H}_d)$ be a pre-homogeneous 
$d$-partition of $K_{rd}^r$. Then  $$rank (\bigoplus_{1\leq i\leq d}\partial_{r-1}^i)=rank (A(\omega_{(\mathcal{H}_1,\mathcal{H}_2,\dots, \mathcal{H}_d)})).$$
\end{lemma}
\begin{proof}
The proof is similar to that for Lemma \ref{lemmaPreH2}.
\end{proof}

\begin{theorem} \label{ThCombR} Let $(\mathcal{H}_1,\mathcal{H}_2,\dots, \mathcal{H}_d)$ be a $d$-partition of $K_{rd}^r$. The following are equivalent
\begin{enumerate} 
\item $det^{S^r}(\omega_{(\mathcal{H}_1,\mathcal{H}_2,\dots, \mathcal{H}_d)})\neq 0$.
\item  $(\mathcal{H}_1,\mathcal{H}_2,\dots, \mathcal{H}_d)$ is pre-homogeneous, and  $b_k(\mathcal{H}_i)=0$ for every $1\leq i\leq d$, and $-1\leq k\leq r-1$.
\item  $(\mathcal{H}_1,\mathcal{H}_2,\dots, \mathcal{H}_d)$ is pre-homogeneous, and  $b_{r-1}(\mathcal{H}_i)=0$ for every $1\leq i\leq d$.
\end{enumerate}
\end{theorem}
\begin{proof} 
If  $det^{S^r}(\omega_{(\mathcal{H}_1,\mathcal{H}_2,\dots \mathcal{H}_d)})\neq 0$, then by Lemma \ref{lemmaPreH1R} we know that the partition $(\mathcal{H}_1,\mathcal{H}_2,\dots, \mathcal{H}_d)$  must be pre-homogeneous. In particular it follows that  for each $1\leq i\leq d$ the $r-2$-skeleton of $X(\mathcal{H}_i)$ coincide with the $r-2$-skeleton of $\Delta_{rd-1}$, and so we get that $b_{k-1}(\mathcal{H}_i)=0$ for all $0\leq k\leq r-2$ and $1\leq i\leq d$. 

Since $(\mathcal{H}_1,\mathcal{H}_2,\dots, \mathcal{H}_d)$  is pre-homogeneous, to compute the homology of $X(\mathcal{H}_i)$ we can use the complex $\mathcal{K}_i$.

Moreover, because  $det^{S^3}(\omega_{(\mathcal{H}_1,\mathcal{H}_2,\dots, \mathcal{H}_d)})\neq 0$, it follows by Lemma \ref{lemmaPreH2R}
that $$rank (\bigoplus_{1\leq i\leq d}\partial_{r-1}^i)=rank (A(\omega_{(\mathcal{H}_1,\mathcal{H}_2,\dots, \mathcal{H}_d)}))=d\binom{rd-1}{r-1},$$ in other words the map $\bigoplus_{1\leq i\leq d}\partial_{r-1}^i$ is one to one.  This implies that each map $$\partial_{r-1}^i:\bigoplus_{\{a_1,\dots,a_r\}\in E_{r}(\mathcal{H}_i)}kf_{i}^{a_1,\dots,a_r}
\to \bigoplus_{1\leq a_1<\dots<a_{r-1}\leq rd}kf_{i}^{a_1,\dots,a_{r-1}}$$ is one to one, in particular we get that $b_{r-1}(\mathcal{H}_i)=0$. 

Finally, notice that if we take $\mathcal{K}$ the direct sum $\bigoplus_{1\leq i\leq d}(\mathcal{K}_i)$ we get that the 
 Euler characteristic is 
\begin{eqnarray*}
\chi_{\mathcal{K}}&=&d(\sum_{k=0}^{r-1}(-1)^k\binom{rd}{k})+(-1)^{r}\binom{rd}{r},
\end{eqnarray*}
and so, by Lemma \ref{lemmaCombrd}, we have  $\chi_{\mathcal{K}}=0$. 

On the other hand, we can use Betti numbers to compute 
$$\chi_{\mathcal{K}}=\sum_{i=1}^d(\sum_{k=0}^{r}(-1)^kb_{k-1}(\mathcal{H}_i))).$$ 
We already proved that  $b_{k-1}(\mathcal{H}_i)=0$ for all $1\leq i\leq d$ and $k\neq r-1$ (i.e. $k-1\neq r-2$), and so we get $$\sum_{i=1}^db_{r-2}(\mathcal{H}_i)=0.$$ Since the Betti numbers are positive we have have $b_{r-2}(\mathcal{H}_i)=0$ for all $1\leq i\leq d$. 

The converse is similar to the proof for Theorem \ref{ThComb}. 
\end{proof}

\begin{corollary} Let $(\mathcal{H}_1,\mathcal{H}_2,\dots, \mathcal{H}_d)$ be a $d$-partition of $K_{rd}^r$ such that  $det^{S^r}(\omega_{(\mathcal{H}_1,\mathcal{H}_2,\dots, \mathcal{H}_d)})\neq 0$. Then $(\mathcal{H}_1,\mathcal{H}_2,\dots, \mathcal{H}_d)$ is a homogeneous $d$-partition.
\end{corollary} 
\begin{proof} From Theorem  \ref{ThCombR}  we know that  $(\mathcal{H}_1,\mathcal{H}_2,\dots, \mathcal{H}_d)$ is pre-homogeneous and $b_k(\mathcal{H}_i)=0$ for all $1\leq i\leq d$, and $-1\leq k\leq r-1$. In particular we have that the Euler characteristic $\chi_{\mathcal{H}_i}$ is zero. 

Since the partition $(\mathcal{H}_1,\mathcal{H}_2,\dots, \mathcal{H}_d)$  is pre-homogeneous we have that  $|E_k(\mathcal{H}_i)|=\binom{rd}{k}$ for all $0\leq k\leq r-1$. We can compute the Euler characteristics $\chi_{\mathcal{K}_i}$ as 
$$\chi_{\mathcal{K}_i}=\sum_{k=0}^{r-1}(-1)^k\binom{rd}{k}+(-1)^r|E_r(\mathcal{H}_i)|.$$
By Lemma \ref{lemmaCombrd} we get $$|E_r(\mathcal{H}_i)|=\frac{1}{d}\binom{rd}{r}=\binom{rd-1}{r-1},$$ in other words $(\mathcal{H}_1,\mathcal{H}_2,\dots, \mathcal{H}_d)$ is a homogeneous $d$-partition.
\end{proof}

\subsection{Some Remarks}

\begin{remark} As mentioned at the beginning of this section, it is not clear if the map $det^{S^r}$ is always nontrivial. In the Appendix we introduce an element $E_d^{(r)}\in V_d^{\otimes\binom{rd}{r}}$, and using MATLAB we check that in certain particular cases $det^{S^r}(E_d^{(r)})\neq 0$. We believe that this is always the case, but do not have a proof yet.  
\end{remark}
 
\begin{remark} 
The uniqueness (up to a constant) of the linear map $det^{S^r}:V_d^{\otimes\binom{rd}{r}}\to k$ with the property that $det^{S^r}((v_{i_1,\ldots,i_r})_{1\leq i_1<i_2<\ldots<i_r\leq rd})=0$ if there exists $1\leq x_1<x_2<\ldots<x_{r+1}\leq rd$ such that
 $$v_{x_1,x_2,\ldots,x_r}=v_{x_1,\ldots,x_{r-1},x_{r+1}}=v_{x_1,\ldots,x_{r-2},x_r,x_{r+1}}=\ldots=v_{x_1,x_3,\ldots,x_r,x_{r+1}}=v_{x_2,x_3,\ldots,x_{r+1}},$$ 
was conjectured in \cite{sta2} for $r=2$, and in \cite{S3} for general $r$. At this point uniqueness is known only for $(r,d)\in\{(2,2),(2,3),(3,2)\}$. 
\end{remark}

\begin{remark}
The maps $det^{S^r}$ are generalizations of the determinant and have connections to hypergraphs. However, while similar in spirit, they are not related to the  hyperdeterminant construction from \cite{gkz}. 
\end{remark}

\begin{remark} It is well known that the determinant map can be written as 
$$det(\otimes_{1\leq i\leq d}(v_i))=\sum_{\sigma \in S_d}\varepsilon(\sigma)\prod_{s=1}^dv^{s}_{\sigma(s)},$$
where  $S_d$ is the symmetric group, $\varepsilon:S_d\to \{-1,1\}$ is the signature map, and  $v_i=v_i^1e_1+v_i^2e_2+\dots+v_i^de_d\in V_d$. 

One can identify $S_d$ with the set of homogeneous $d$-partitions of the complete $1$-uniform hypergraph $K_d^1$. More precisely, we associate to every permutation  $\sigma\in S_d$ the homogeneous $d$-partitions $(\Pi_1,\Pi_2,\dots,\Pi_d)=(\{\sigma(1)\}, \{\sigma(2)\},\dots, \{\sigma(d)\})\in \mathcal{P}_{d}^{h}(K_{d}^1)$. With this identification, the above formula can be rewritten as
$$det(\otimes_{1\leq i\leq d}(v_i))=\sum_{(\Pi_1,\dots,\Pi_d)\in \mathcal{P}_{d}^{h}(K_{d}^1)}\varepsilon((\Pi_1,\dots,\Pi_d))\prod_{s=1}^d (\prod_{i\in \Pi_s} v_{i}^{s}).$$

It is known from \cite{sta2} that if the map $det^{S^2}$  exists and it is unique (with the above mentioned property), then one can write it as 
$$det^{S^2}(\otimes_{1\leq i<j\leq 2d}(v_{i,j}))=\sum_{(\Gamma_1,\dots,\Gamma_d)\in \mathcal{P}_{d}^{h,cf}(K_{2d}^2)}\varepsilon^{S^2}_d((\Gamma_1,\dots,\Gamma_d))\prod_{s=1}^d (\prod_{(i,j)\in \Gamma_s} v_{i,j}^{s}),$$
where $\mathcal{P}_{d}^{h,cf}(K_{2d}^2)$ is the set of homogeneous cycle free $d$-partition of the complete graph $K_{2d}^2=K_{2d}$, $\varepsilon_d^{S^2}:\mathcal{P}_{d}^{h,cf}(K_{2d}^2)\to \{-1,1\}$, and  $v_{i,j}=v_{i,j}^1e_1+v_{i,j}^2e_2+\dots+v_{i,j}^de_d\in V_d$.  

The analog  statement for $r\geq 3$ is the following. Suppose that there exists a unique (up to a constant) linear map $det^{S^r}:V_d^{\otimes\binom{rd}{r}}\to k$ with property that $det^{S^r}(\otimes_{1\leq i_1<i_2<\ldots<i_r\leq rd}(v_{i_1,\ldots,i_r}))=0$ if there exists $1\leq x_1<x_2<\ldots<x_{r+1}\leq rd$ such that
 $$v_{x_1,x_2,\ldots,x_r}=v_{x_1,\ldots,x_{r-1},x_{r+1}}=v_{x_1,\ldots,x_{r-2},x_r,x_{r+1}}=\ldots=v_{x_1,x_3,\ldots,x_r,x_{r+1}}=v_{x_2,x_3,\ldots,x_{r+1}}.$$ Then 
$$det^{S^r}(\otimes_{1\leq i_1<\dots <i_r\leq rd}(v_{i_1,\dots,i_r}))=\sum_{(\mathcal{H}_1,\dots,\mathcal{H}_d)\in \mathcal{P}_{d}^{h,zB}(K_{rd}^r)}\varepsilon^{S^r}_d((\mathcal{H}_1,\dots,\mathcal{H}_d))\prod_{s=1}^d (\prod_{(i_1,\dots,i_r)\in \mathcal{H}_s} v_{i_1,\dots,i_r}^{s}),$$
where $\mathcal{P}_{d}^{h,zB}(K_{rd})$ is the set of homogeneous $d$-partition of the complete $r$-uniform hypergraph $K_{rd}^r$ that have zero Betti numbers, $\varepsilon_d^{S^r}:\mathcal{P}_{d}^{h,zB}(K_{rd}^r)\to \mathbb{Z}^*$, and  $v_{i_1,\dots,i_r}=v_{i_1,\dots,i_r}^{1}e_1+v_{i_1,\dots,i_r}^{2}e_2+\dots+v_{i_1,\dots,i_r}^{d}e_d\in V_d$.  

Notice that for $r\geq 3$ the  map $\varepsilon_d^{S^r}$ does not take only the values $\pm 1$. For example, it was show in  \cite{S3} that if $d=2$ and  $r=3$ then the map $det^{S^3}:V_2^{\otimes 20}\to k$ exists and it is unique, and the corresponding function is $\varepsilon_2^{S^3}:\mathcal{P}_{2}^{h,zB}(K_{6}^3)\to \{-4,-1,1\}$. 
\end{remark}
\begin{remark}
It is reasonable to ask what is the interpretation of the  map $\varepsilon_d^{S^r}$. For the case $r=2$ this was briefly discussed in \cite{edge}. It is likely that the value of  $\varepsilon^{S^r}_d((\mathcal{H}_1,\dots,\mathcal{H}_d))$ is related to the orders of torsion groups of $X(\mathcal{H}_i)$. This is a problem that we plan to investigate in a future paper. 
\end{remark}

\appendix

\maketitle

\section{Cases for which the map $det^{S^r}$ is nontrivial (Joint work with Tony Passero)}

In  this section we show that the map $det^{S^r}:V_d^{\otimes\binom{rd}{r}}\to k$ is nontrivial for certain values of $r\geq 4$ and $d\geq 2$.  These results were obtained  using MATLAB. First we need to introduce an element $E^{(r)}_d \in V_d^{\otimes\binom{rd}{r}}$. 

For every  $1\leq a\leq d$ take $${\bf S}_a = \{ra - (r-1),ra-(r-2), \dots, ra\}.$$ Let $1 \le i_1 < i_2 < \cdots < i_r \le rd$, for each $1\leq k\leq r$ we take $a_k\{1,2,\dots,d\}$ such that $i_k \in {\bf S}_{a_k}$. Notice that $1\leq a_1\leq a_2\leq \dots \leq a_r\leq d$. Define 
$$e_{i_1,i_2,\dots,i_r}= e_{a_{t}},$$ where $t\in\{1,2,\dots,r\}$ and 
$$t-1 = \sum \limits_{j=1}^r i_j \; (mod \; r).$$ With this notation we define
 $$E^{(r)}_d = \otimes_{1\leq i_1<i_2<\dots <i_r\leq rd}(e_{i_1,i_2,\dots,i_r}) \in V_d^{\otimes \binom{rd}{r}}.$$ 

\begin{example} If we take $r=5$ and $d=6$ we have $e_{1,2,7,14,28}=e_2$. That is because  because $1+2+7+14+28=52=2=3-1\; (mod \; 5)$ (i.e. $t=3$), and $i_3=7\in {\bf S}_{2}$ (i.e. $a_t=a_3=2$). 
Similarly, $e_{26,27,28,29,30}=e_6$. That is because $26+27+28+29+30=0=1-1\; (mod \; 5)$ (i.e. $t=1$), and $i_1=26\in {\bf S}_6$ (i.e. $a_t=a_1=6$).
\end{example}

\begin{proposition} In Table \ref{tablerd} we have a list of $(r,d)$ for which $det^{S^r} (E^{(r)}_d) \neq 0$. In particular, for those values $(r,d)$ we get that $det^{S^r}$ is nontrivial, and so $dim_k(\Lambda_{V_d}^{S^{r}}[rd])\geq 1$. 
\end{proposition}
\begin{proof} A MATLAB script was created to generate the matrix corresponding to the system of equations $\mathcal{S}(E^{(r)}_d)$ described in \ref{relEqSr}. Using the definition of $det^{S^r}$ we calculated  $det^{S^r}(E^{(r)}_d)$. The values obtained are presented in Table \ref{tablerd}.

\begin{table}[hbt!]
        \centering
        \label{determinant}
        \begin{tabular}{|c|c|c|c|c|c|c|c|c|c|c|c|}
            \cline{3-11}
            \multicolumn{2}{c}{}& \multicolumn{9}{|c|}{$d$}\\
            \cline{3-11}
            \multicolumn{2}{c|}{}& 2 & 3 & 4 & 5 & 6 & 7 & 8 & 9 &10\\
            \hline
            \multirow{7}{*}{$r$} & 2 & -1 & 1 & 1 & 1 & -1 & 1 & 1 & 1 & -1\\
            \cline{2-11}
            & 3 & -1 & -1 & 1 & 1 & 1 & 1 & 1 & 1 & -1\\
            \cline{2-11}
            & 4 & 1 & -1 & 1 & 1 & 1 & -1 & 1 & 1 & -\\
            \cline{2-11}
            & 5 & -1 & -1 & 1 & 1 & - & - & - & - & -\\
            \cline{2-11}
            & 6 & 1 & -1 & - & - & - & - & - & - & -\\
            \cline{2-11}
            & 7 & 1 & - & - & - & - & - & - & - & -\\
            \cline{2-11}
            & 8 & 1 & - & - & - & - & - & - & - & -\\
            \hline
        \end{tabular}
				    \vspace{1em}
				        \caption{Values of $det^{S^r}(E^{(r)}_d)$}\label{tablerd}
\end{table}

\end{proof}

\begin{remark} The empty spots in Table \ref{tablerd} are due to limited computing resources. For example, if $r=5$ and $d=6$ one has to compute the determinant of a square matrix of size \num{142506}. This is more than one can do on a regular laptop. 
\end{remark}

\begin{remark}  We expect that $det^{S^r}(E^{(r)}_d)\neq 0$ for all $r$ and $d$, however we do not have a proof for this statement. One possible approach would be to use the result from Theorem \ref{ThCombR}. It is easy to see that the $d$-partition associated to $E^{(r)}_d$ is homogeneous. In order to get the general result one has to prove that its Betti numbers are zero.  
\end{remark}

\section*{Declarations of interest} 
There is no conflict of interest.



\section*{Acknowledgment}
We thank Ben Ward for comments on an earlier version of this paper. Some of the results from this paper are part of the first author's Ph.D. thesis at Bowling Green State University.

\bibliographystyle{amsalpha}

\begin{thebibliography}{A}


\bibitem 
{bretto}
A. Bretto, \textit{Hypergraph theory: an introduction}. Springer, New York, (2013).




\bibitem
{gkz}
I. M. Gelfand, M. M. Kapranov,  and A. V. Zelevinsky,    \textit{Discriminants, resultants, and multidimensional determinants}. Birkhäuser Boston,  (1994).


\bibitem{geo}
A. Georgakopoulos,  J. Haslegrave,  R. Montgomery, and B. Narayanan, \textit{Spanning surfaces in 3-graphs}. J. Eur. Math. Soc. \textbf{24} (2022), no. 1,  303--339.

\bibitem{gk}
G. Y. Katona,  and H. A. Kierstead, \textit{Hamiltonian chains in hypergraphs}. J. Graph Theory, \textbf{30} (1999), no. 3, 205--212.



\bibitem 
{edge} S. R. Lippold, M. D. Staic, and A. Stancu, \textit{Edge partitions of the complete graph and a determinant-like function}. Monatsh. Math. \textbf{198} (2022), 819--858. 


\bibitem{rrs}
V. R\"{o}dl, A. Ruci\'{n}ski, and E. Szemer\'{e}di, 
\textit{Dirac-type conditions for Hamiltonian paths and cycles in 3-uniform hypergraphs}. Adv. Math. \textbf{227} (2011), no. 3, 1225--1299.



\bibitem{seb}
V. T. S\'{o}s,   P. Erd\"{o}s, and W. G. Brown, 
\textit{On the existence of triangulated spheres in 3-graphs, and related problems.}
Period. Math. Hungar. {\bf 3} (1973), no. 3-4, 221--228.

\bibitem
{sta2} M. D. Staic, \textit{The Exterior Graded Swiss-Cheese Operad $\Lambda^{S^2}(V)$ (with an appendix by Ana Lorena Gherman and Mihai D. Staic)}. 	 Comm. Algebra, {\textbf 51} (2023), no. 7, 2705--2728.



\bibitem{dets2} M. D. Staic, \textit{Existence of the $det^{S^2}$ Map}.  To appear in Bulletin of the London Mathematical Society, arXiv:2205.02178.

\bibitem{S3} 
M. D. Staic, and S. R. Lippold,  \emph{Partitions of the complete hypergraph $K_6^3$ and a determinant-like function.} J Algebr. Comb., \textbf{56} (2022), 969--1003.

\bibitem{sv}
M. D. Staic, and J. Van Grinsven, \textit{A geometric application for the map $det^{S^2}$}. Comm. Algebra, \textbf{50} (2022), no. 3, 1106--1117.

\bibitem
{stu} B. Sturmfels, \textit{Algorithms in invariant theory}. Springer Wien New York, (2008).

\end{thebibliography}

 \end{document}